\documentclass[a4paper]{scrartcl}

\usepackage{amsfonts, dsfont}
\usepackage{amsmath}
\usepackage{mathtools}
\usepackage{amsthm}
\usepackage{color}

\newtheorem{lemma}{Lemma}

\newtheorem{assumption}{Assumption}
\newtheorem{theorem}[lemma]{Theorem}
\newtheorem{remark}{Remark} 
\renewcommand{\vec}{\boldsymbol}

\newcommand{\vertiii}[1]{{\left\vert\kern-0.25ex\left\vert\kern-0.25ex\left\vert #1 
		\right\vert\kern-0.25ex\right\vert\kern-0.25ex\right\vert}}

\usepackage{geometry}
\usepackage[headsepline=.2pt]{scrlayer-scrpage}
    
    \addtokomafont{pagehead}{\slshape}
  \ihead{Frei, Judakova, Richter: Locally modified second-order FEM for interface problems}

\usepackage[square,numbers,sort&compress]{natbib}

\makeatletter
\renewcommand\NAT@citesuper[3]{\ifNAT@swa
\if*#2*\else#2\NAT@spacechar\fi
\unskip\kern\p@\textsuperscript{\NAT@@open#1\if*#3*\else,\NAT@spacechar#3\fi\NAT@@close}%
   \else #1\fi\endgroup}
\makeatother

\usepackage{amsfonts, dsfont}
\usepackage{amsmath}
\usepackage{mathtools}
\usepackage{amsthm}
\usepackage{stmaryrd}
 
\usepackage{framed}
\usepackage{multicol}

\usepackage{color}
\usepackage{array}
\usepackage{booktabs}
\usepackage[final]{graphicx}
\usepackage{enumitem}
\usepackage{url}
\usepackage{varwidth}
\usepackage{hhline}

\usepackage{mathtools}

\usepackage[
  normal,
  font={small,color=black},
  labelfont=bf,
  margin=2.5em,
  labelsep=period]{caption}[2008/04/01]

\usepackage{caption}

\renewcommand{\kappa}{\varkappa}

 \renewcommand{\vec}[1]{\mathbf{#1}}

 \providecommand{\keywords}[1]{\textbf{Keywords}  #1}

\begin{document}

\title{A locally modified second-order finite element method for interface problems {and its implementation in 2 dimensions}}

\date{}

%
\author{Stefan Frei\thanks{Department of Mathematics \& Statistics, University of Konstanz, Germany (stefan.frei@uni-konstanz.de)} 
\and Gozel Judakova\thanks{Institut f\"ur Analysis und Numerik, University of Magdeburg, Germany, (gozel.judakova@ovgu.de)}
\and Thomas Richter\thanks{Institut f\"ur Analysis und Numerik, University of Magdeburg, Germany (thomas.richter@ovgu.de)}}
\date{}

\maketitle

\abstract{
The locally modified finite element method, which is introduced in \cite{Frei1}, is 
a simple fitted finite element method that is able to resolve weak discontinuities in interface problems. The method is based on 
a fixed structured coarse mesh, which is then refined into sub-elements to resolve an interior interface.
 In this work, we extend the locally modified finite element method {in two space dimensions} to second order using an isoparametric approach in the 
 interface elements. Thereby we need to take care that the resulting curved edges do not lead to degenerate sub-elements. We prove optimal a priori error estimates in the $L^2$-norm and in a {discrete} energy norm.
Finally, we present numerical examples to substantiate the theoretical findings.\\}


\section{Introduction}
In this paper, we extend the locally modified finite element method introduced in \cite{Frei1, RichterBook, FreiDiss} to higher order. We investigate 
interface problems, where the normal derivative of the solution may have a jump over an interior interface. 
{Examples of such interface problems are ubiquitous in technology, industry, science and medicine. Some of the most prominent examples include fluid-structure interactions~\cite{RichterBook, BazilevsTakizawaTezduyarBuch} or multiphase flows~\cite{GrossReusken}. Fluid-structure interactions arise in aerodynamical applications like flow around airplanes or parachutes~\cite{Steinetal2000}, in biomedical problems such as blood flow through the cardiovascular system~\cite{Peskin1972, VandeVosseetal2003, Formaggiaetal2010} or the airflow within the respiratory system~\cite{WallRabczuk2008} and even in tribological applications~\cite{KnaufFreiRichterRannacher}. Multiphase problems include 
gas-liquid and particle-laden gas flows, rising bubbles~\cite{Hysingetal2009}, droplets in microfluidic devices~\cite{ClausKerfriden2019} or the simulation of tumor growth~\cite{Garckeetal2018}.
Another field of application are shape or topology optimization problems including multi-component structures~\cite{GanglLanger, BurmanHeLarson2021}. The simplest possible setting, which is the scope of the present paper, is a diffusion problem where the coefficient is discontinuous across an interior interface.
}


We assume that the domain {$\Omega \subset \mathbb{R}^2$} is divided into $\Omega = \Omega_1 \cup \Gamma \cup \Omega_2$ with {an interface $\Gamma\subset \partial \Omega_1 \cap \partial \Omega_2$, such that $\overline{\Gamma}= \partial \Omega_1 \cap \partial \Omega_2$,}
and a discontinuous diffusion coefficient $\nu>0$ across $\Gamma$. In order to simplify the analysis we will assume that the outer domain $\Omega$ is a two-dimensional convex domain with polygonal boundary. Some remarks on corresponding three-dimensional methods will be given in {Remark~\ref{rem.3d}.} We consider the equations
\begin{align}
  -\nabla \cdot (\nu_i \nabla u) &= f \quad { \text{in} \; \Omega_i}, \;\; i=1,2,\label{laplace}\\
  [u] = 0,\; [\nu \partial_n u] &= 0\quad\text{on} \; \Gamma,\label{interfaceNormal} \\
  u &= 0\quad \text{on} \; \partial \Omega,
\end{align}
where $\nu|_{\Omega_i}:=\nu_i$, $i=1,2$ and
the jump $[w]$ at the interface $\Gamma$ with normal vector $\vec{n}$ is defined by
\[
  [w]({\vec{x}}):=\lim_{s \searrow 0} w({\vec{x}+s\vec{n}}) - \lim_{s \nearrow 0} w({\vec{x}+s\vec{n}}), \quad {\vec{x}}\in \Gamma.
\]
The corresponding variational formulation of the problem (\ref{laplace}) is given by
\begin{align}\label{VarForm}
  u\in H_0^1(\Omega): \; \; \sum_{i=1}^{2}(\nu_i \nabla u, \nabla \varphi)_{\Omega_i} = (f,\varphi)_\Omega \; \; \, \forall \varphi\in H_0^1(\Omega). 
\end{align}
This interface problem is intensively discussed in the literature. Babu\v{s}ka \cite{Babuska} shows that a standard finite element ansatz {has} low accuracy, 
regardless of the polynomial degree of the finite element space,
$$||u-u_h||_{\Omega} = \mathcal{O}(h),\quad ||\nabla (u-u_h)||_{\Omega} = \mathcal{O}(h^{1/2}),$$
{where throughout the paper $\|\cdot\|_\Omega := \|\cdot\|_{L^2(\Omega)}$ denotes the $L^2$-norm.}
To improve the accuracy, the interface needs to be resolved within the discretization.  
Frei and Richter \cite{Frei1} presented a locally modified finite element method based on first-order polynomials with first-order accuracy in the energy norm and second order in the $L^2$-norm.
The method is based on a fixed coarse \emph{patch mesh} consisting of quadrilaterals, which is independent of the position of the interface. 
The patch elements are then divided into sub-elements, such that the interface is locally resolved. The discretization is based on piecewise 
linear finite elements which 
has a natural extension to higher order finite element spaces. 

{Due to the fixed background \textit{patch mesh} this approach is particularly suitable for problems involving moving interfaces, where functions $u_h(t_{n-1})$ and $u_h(t_n)$ defined on different sub-meshes need to be integrated against each other within a time-stepping scheme~\cite{FreiRichter2017}. Due to the implicit adaption of the finite element spaces within the locally modified finite element method, a costly re-meshing procedure is avoided. Similarly, the locally modified finite element method might be useful in shape or topology optimization problems, where problems need to be solved for different interface and boundary positions, while approaching the solution~\cite{BurmanHeLarson2021, GanglLanger}.
}

The \emph{locally modified finite element method} has been used by the authors and 
co-workers~\cite{FreiRichterWick_Growth,FreiRichterWick_enumath1,
FreiRichterWick_enumath2,FreiDiss},
and by Langer \& Yang~\cite{LangerYang} for fluid-structure interaction (FSI) problems, including the 
transition from FSI to solid-solid contact~\cite{BurmanFernandezFrei, BurmanFernandezFreiGerosa, FreiRichter2017Sammelband}. 
Holm et al.~\cite{Holmetal} and Gangl \& Langer~\cite{GanglLanger} used a
corresponding approach based on triangular patches, the latter work being motivated by a topology optimization problem. A pressure stabilization 
technique for flow problems has been developed
in~\cite{FreiPressure} and a suitable (second-order) time discretization scheme in~\cite{FreiRichter2017}. 
Details {of} the implementation in deal.ii 
and the corresponding source code have 
been published in~\cite{FreiRichterWick_impl, zenodo}.
Extensions to three space dimensions have been developed by 
Langer \& Yang~\cite{LangerYang2015}, where {hexahedral} coarse cells are divided into sub-elements consisting of {hexahedra} and tetrahedra, and by H\"ollbacher \& Wittum, where a coarse mesh consisting of tetrahedra is sub-divided into hexahedrons, prisms and pyramids~\cite{HoellbacherWittum, Vogeletal2013}.

Alternative approaches are \text{unfitted methods}, where the mesh is fixed and does not resolve the interface. Prominent examples are the extended 
finite element method {(XFEM~\cite{MoesDolbowBelytschko1999, daux2000arbitrary,chessa2003extended,fries2010extended})} and the generalised finite element (GFEM~\cite{BabuskaBanarjeeOsborn2004}), 
{where the finite element space is enriched by suitable functions that contain certain properties of the solution (for example discontinuities in the function or its derivative).
Higher-order approximations within the XFEM approach have been developed by Cheng \& Fries~\cite{ChengFries2010} and by Dr{\'e}au, Chevaugeon \& Mo{\"e}s~\cite{Dreauetal2010}.
A conceptionally different unfitted approach are Cut Finite Elements (CutFEM)~\cite{HansboHansbo2002, BurmanHansbo,hansbo2014cut, CutFEM, Zahedi2017}. Here the main difficulty lies in the construction of quadrature formulas to represent the interface in the cut cells. Possibilities to obtain higher-order approximations include the definition of parametric mappings in the cut cells~\cite{Lehrenfeld1, Lehrenfeld2, Lehrenfeld3} or a boundary value correction based on Taylor expansion~\cite{burman2018cutmoc}. Unfitted discontinuous Galerkin methods within the CutFEM paradigm have been developed in~\cite{FidkowskiDarmofal2007, BastianEngwer2009, Massjung2012}. Areias \& Belytschko noted that CutFEM and the discontinuous variant of the XFEM approach are in fact closely related~\cite{AreiasBelytschko2006}.
A further unfitted approach that circumvents the problem to construct quadrature formulas is the shifted boundary method~\cite{MainScovazzi2018}, where interface conditions are formulated on neighbouring edges instead of the interface $\Gamma$.}

For further \textit{fitted} finite element methods, we refer to
\cite{Babuska1970, BastingPrignitz2013,
BrambleKing1996, FeistauerSobotikova1990, Zenisek1990}. 
Some works are similar to the \textit{locally modified finite element methods} in the sense that only mesh elements close to the
interface are altered~\cite{Boergers1990, XieItoLiToivanen2008}. \textit{Fitted} methods with higher order approximations have been
developed by Fang~\cite{Fang2013} {and by Omerovi\'c \& Fries~\cite{OmerovicFries}. Recently, a method called \textit{Universal Meshes} gained certain popularity~\cite{RangarajanLew}. Here the idea is to construct a suitable mapping for the elements in the interface region to resolve the interface.}\\

After this introduction, we describe the locally modified high order finite element approach and show a maximum angle 
condition in Section~\ref{sec.lmfem}. In Section~\ref{sec.apriori}, we derive the main results of this work, namely a priori error estimates in the $L^2$- and in a {discrete} energy norm. 
Section~\ref{sec.impl} gives some details on the implementation
 and in Section~\ref{sec.num}, we show different numerical examples. The conclusion of our work follows in Section $6$.

\section{Locally modified high order finite element method}
\label{sec.lmfem}
In this section we review the first order approach proposed by Frei and Richter \cite{Frei1} and extend it to a second order discretization. The splitting into subelements that
we propose is slightly different from the one presented in \cite{Frei1} and leads in general to a better bound for the maximal angles within the triangles.

Let ${\cal T}_P$ be a form and shape-regular quadrilateral mesh of the {convex domain $\Omega$ with polygonal boundary}. {The elements $P\in {\cal T}_P$ are called patches} and these do not necessarily resolve the partitioning. {By a slight abuse of notation, we will in the following write $P$ both for the elements of the triangulation and for the domain spanned by a patch $P$.} The interface $\Gamma$ may cut the patches. In this case we make {the} assumption:
\begin{assumption}[Interface configuration]\label{ass.allowed}\hfill
  \begin{enumerate}
  \item Each patch $P\in {\cal T}_P$ is either cut $P\cap \Gamma \neq {\emptyset}$ or not cut $P\cap \Gamma ={\emptyset}$. If it is cut, the cut goes through exactly two points on the boundary $\partial P$, 
    see Figure \ref{fig:patch} (left and top right).
  \item The interface does not cut the same edge multiple times and may not enter and leave the patch at the same edge, see Figure \ref{fig:patch} (bottom right).
  \end{enumerate}
\end{assumption}
{Given a smooth interface $\Gamma$, this assumption is fulfilled after sufficient refinement.}
The patch mesh ${\cal T}_P$ is the fixed background mesh used in the parametric finite element method described below. We will introduce a further local refinement of the mesh, denoted ${\cal T}_h$, {which resolves} the interface. However, this refined mesh is only for illustrative purpose. The numerical realization is based on the fixed mesh ${\cal T}_P$ and the "refinement" is in fact only included in a parametric way within the reference map for each patch $P\in {\cal T}_P$. 

\begin{figure}[t]
    \centering
  \includegraphics[width=0.28\textwidth]{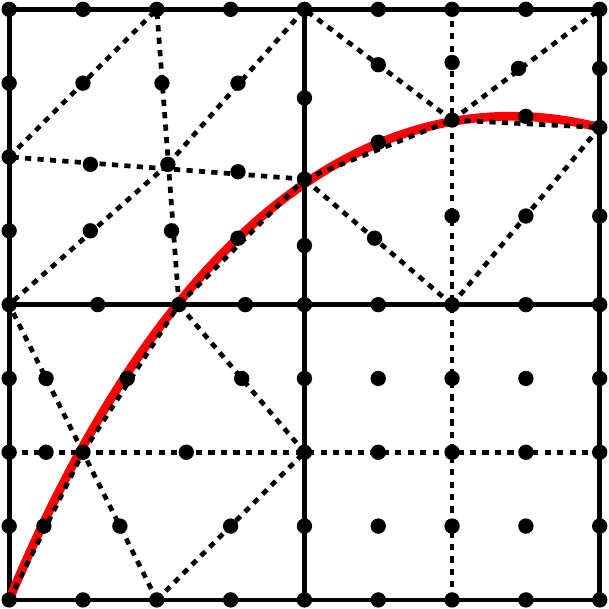}
  \hspace{0.05\textwidth}
  \includegraphics[width=0.52\textwidth]{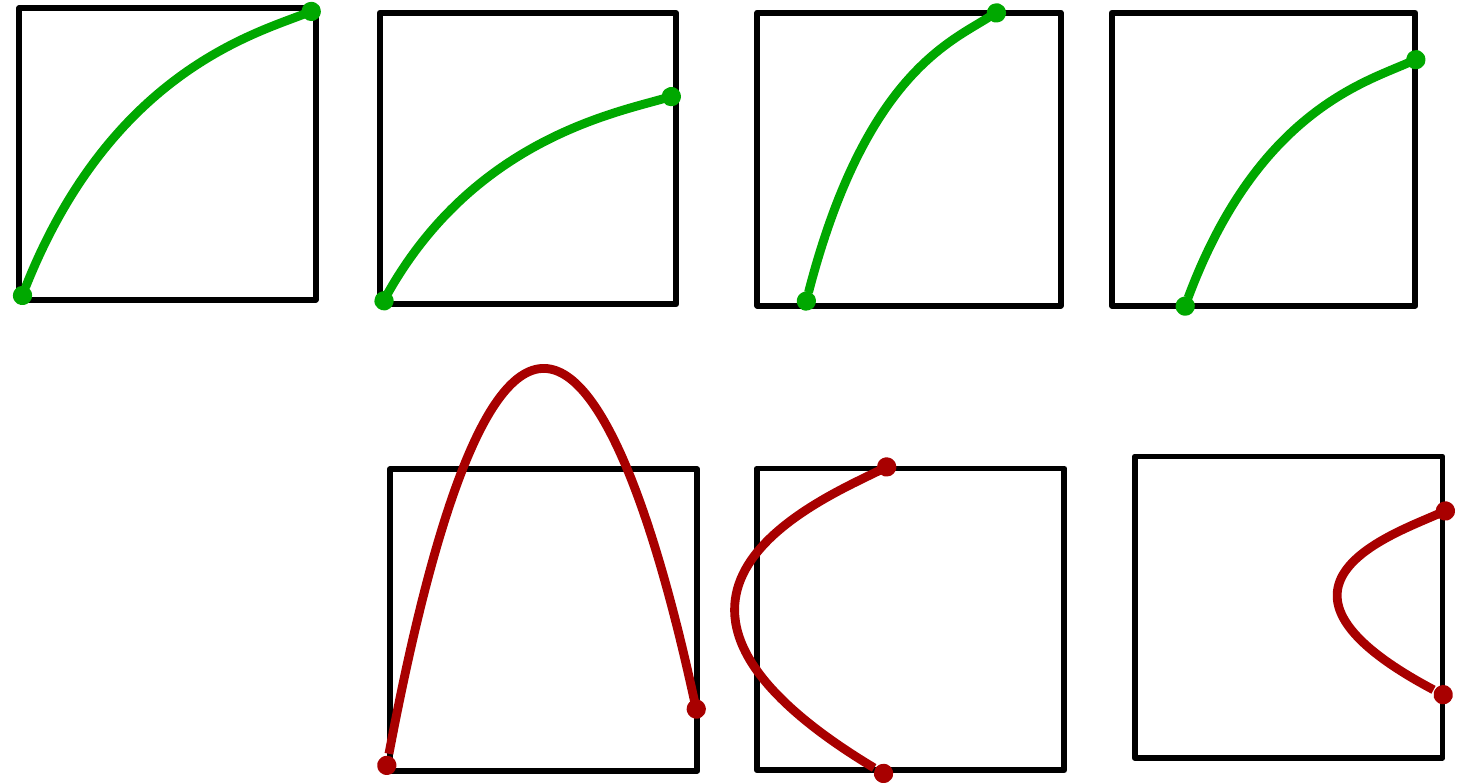}

  \caption{\textit{Left}: Mesh consisting of four patches, three of which are cut by the interface. \textit{Right}: possible configurations (top), and configurations that are not allowed (bottom).}
  \label{fig:patch}
\end{figure}

If the interface is matched by one of the edges of the patch, then the patch is considered as not cut. We will split such patches into four quadrilaterals. 
If the interface cuts the patch, then the patch splits either in eight triangles or in four quadrilaterals.
In both {cases}, the patch $P$ is first split into four quadrilaterals, {denoted by $K_1,\ldots,K_4$}, which are then possibly refined into two triangles each. {The resulting sub-cells are denoted by $T_1,\ldots, T_8$ in the case of triangular sub-cells and by $T_1,\ldots,T_4$ in the case of quadrilaterals (In the latter case these are identical to $K_1,\ldots K_4$).} This two-step procedure will simplify the following proofs. 
We define the isoparametric finite element space ${\tilde{V}}_h\subset H_0^1(\Omega)$
\begin{equation}\label{modFEspace}
  {\tilde{V}_h}:=\{\varphi\in C(\Omega) \,\arrowvert\, {(\varphi \circ {\xi_T})\in \mathcal{P}^r_T(\hat{T})\; \text{for} \; T}\in {\cal T}_h\}, 
\end{equation}
where 
\[
{\mathcal{P}^r_T(\hat{T}):= 
\begin{cases} 
  Q_r(\hat{T}), \; \; T \; \; \text{is a quadrilateral},\\
  P_r(\hat{T}), \; \; T \; \; \text{is a triangle},\\                        
\end{cases} }
\]
and {$\xi_T\in \mathcal{P}^1_T(\hat{T})$} is a transformation from the reference element {$\hat{T}$ to $T$}.
The space {$\tilde{V}_h$} is continuous, as the restriction of a function in {$Q_r(\hat{T})$} to a line $e\subset{\partial \hat{T}}$ is in $P_r({\hat{T}})$.

\subsection{Maximum angle condition}
In order to show optimal-order error estimates, the finite element mesh needs to fulfill a maximum angle condition in a fitted finite element method.
We first analyse the maximum angles of the subtriangles in a Cartesian patch grid ${\cal T}_P$. A bound for a general regular patch mesh can be obtained by using 
the regularity of the patch mesh.

\subsubsection{Linear interface approximation}

\begin{figure}[h!]
  \centering
  \includegraphics[trim=50mm 120mm 10mm 20mm,clip,width=\textwidth]{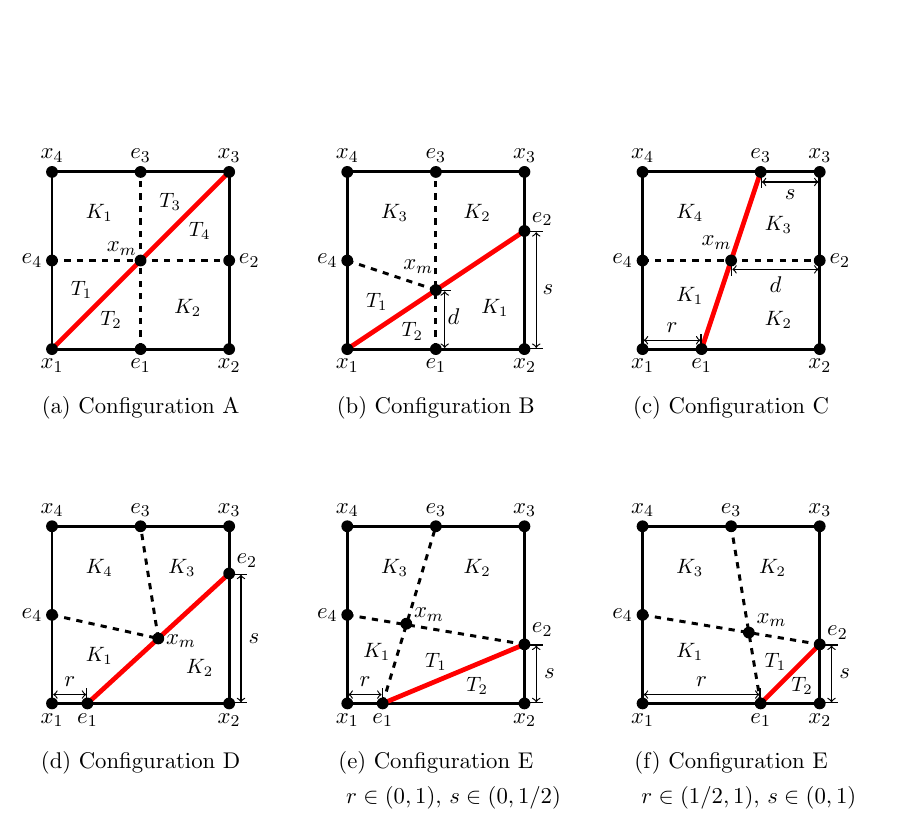}
  \caption{Different configurations and splitting into four large quadrilaterals $K_1,...,K_4$. The red line shows a linear approximation of 
    the interface. In quadrilaterals that are not split by the interface, we divide in such a way into subtriangles, that the largest angle is split. {Note in particular that the degenerate quadrilateral $K_2$ in Configuration D will be split into two regular subtriangles that are used in the definition of finite element spaces.}
    \label{fig:config}  }
\end{figure}

We distinguish {between} five different types of interface cuts by the fact that the interface intersects a patch either in 1 or 2 exterior vertices (Config. B and A) or two opposite (C) or adjacent (D and E) edges, see Figure~\ref{fig:config}.
{Let $r,s \in(0,1)$ denote the relative cut positions on an edge $e$ (see Figure~\ref{fig:config}).  
In the case of adjacent edges, we distinguish} further between the case that $r\leq \frac{1}{2}$ and $s\geq \frac{1}{2}$ (D) and the case that one these inequalities is violated (E).
In all cases the patch element can be split in four large quadrilaterals 
$K_1,...,K_4$ first, which are then {divided} into two sub-triangles, if the interface cuts through the patch. Details are given in the appendix.

Considering arbitrary interface positions, anisotropic elements can arise, when the relative cut position $r,s \in(0,1)$ on an edge $e$
tends to $0$ or $1$ (see Figure~\ref{fig:config}). 
We can not guarantee a minimum angle condition for the sub-triangles, but we can ensure that the maximum 
angles remain bounded away from $180^\circ$.

\begin{lemma}[Linear approximation of the interface]
  \label{linear}
  All interior angles of the Cartesian patch elements shown in Figure \ref{fig:config} are bounded by $135^{\circ}$ independently of the parameters $r,s \in(0,1)$.
\end{lemma}
\begin{proof}
  The proof follows by basic geometrical considerations, see Appendix B.
\end{proof}

\begin{theorem}
  \label{max_angle}
  We assume that the patch grid ${\cal T}_P$ is Cartesian. 
  For all types of interface cuts (see Figure~\ref{fig:config}), the interior angles of all subelements are 
  bounded by $135^{\circ}$ independently of the parameters $r,s \in(0,1)$. 
\end{theorem}
\begin{proof}
  By means of Lemma~\ref{linear} all interior angles on the reference patch are bounded by $135^{\circ}$. As all cells are Cartesian, the same bound holds for the elements {$T\subset P$} .
\end{proof}

\begin{remark}
  We have assumed for simplicity that the underlying patch mesh is fully
  Cartesian. This assumption can, however, easily be weakened. Allowing
  more general form- and shape-regular patch meshes a geometric transformation of each patch
  to the unit patch will give a bound $\alpha<\alpha_{\max}<180^\circ$ for the interior angles $\alpha$ (with $\alpha_{\max}$ larger than $135^\circ$). 
\end{remark}

\subsubsection{Quadratic interface approximation}

Next, we define a quadratic approximation of the interface. In each of the subtriangles obtained in the previous paragraph, we consider 6 degrees of freedom 
that lie on the vertices and edge midpints of the triangles (see the dots in Figure~\ref{fig:patch}, left). In order to guarantee a higher-order interface approximation
those that lie on the discrete interface $\Gamma_h$ need to be moved. The detailed algorithm is given in Section~\ref{sec.impl}.

\begin{figure}[t]
  \centering
  \includegraphics[width=0.8\textwidth]{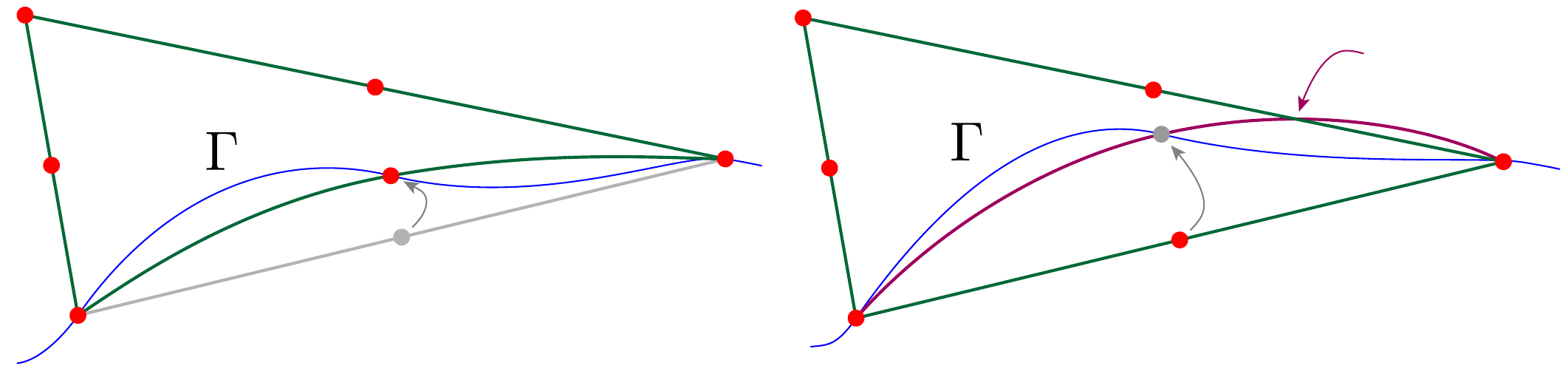}
  \caption{Each of the eight triangles approximates the
    interface $\Gamma$ quadratically. {Herefore}, the midpoint of the
    edge that corresponds to the interface is pulled onto the curve.
    {\textit{Left}}: This is a valid configuration where a quadratic
    approximation is possible. \textit{Right}: In some configurations
    a fully quadratic interface approximation would
    result in a degenerate element with an interface that is leaving
    triangle (see the mark on the upper edge). Such triangles are approximated linearly and cause $n_l>0$.}
  \label{fig:config-step2}
\end{figure}

%
%
In certain ``pathological'' situations we can not guarantee that the angle conditions imposed above are fulfilled. This is due to the fact that the curved edges that correspond to a quadratic interface approximation might intersect other edges, 
see Figure~\ref{fig:config-step2} for an example.

In this case, we use a linear approximation of the interface in the affected patch {(i.e.\,the reference maps $\xi_T$ in the finite element space~\eqref{modFEspace} are linear). We denote the set of patches, where a linear interface approximation is used by ${\cal T}_{P,\text{lin}}$ and the corresponding set of sub-cells with a linear reference map $\xi_T$ by ${\cal T}_{h,\text{lin}}$.} We will see in the numerical examples below that this happens rarely. Moreover, it is reasonable to assume that the maximum number of such patches remains bounded under refinement independently of $h\le h_0$. 

{We give a heuristic argument for this assumption. Let us consider the situation sketched in Figure~\ref{fig:config-step2}. The linear approximation of the interface will never leave the patch by definition. The maximum distance between a linear and quadratic interface approximation is bounded by ${\cal O}(h_P^2)$. In relation with the patch size ${\cal O}(h_P)$ this means that -considering arbitrary interface positions- the probability that the quadratic interface approximation leaves the patch is bounded by ${\cal O}(h_P)$. The number of interface patches, on the other hand, grows like ${\cal O}(h_P^{-1})$. Hence it is reasonable to assume that the number of affected patches behaves like ${\cal O}(1)$ for $h_P\to 0$.}
We will denote the maximum number of patches with a linear interface approximation by $n_l$.


\subsection{Modified spaces and discrete variational formulation}

We {define the finite element space}
\begin{equation}
  V_h:=\{\varphi\in C(\Omega) \,\arrowvert\, {(\varphi \circ {\xi_T})\in \mathcal{P}^r_T(\hat{T})\; \text{for} \; T}\in {\cal T}_h\}, 
\end{equation}
where the map {$\xi_T$} resolves the interface with order $r$ in all but $n_l$ elements, where the approximation is only linear:
{\begin{align*}
\xi_T \in \begin{cases}
{\cal P}_T^1(\hat{T}), \qquad &T\in {\cal T}_{h,\text{lin}},\\
{\cal P}_T^r(\hat{T}), \qquad &\text{else}.
\end{cases}
\end{align*}}
 The polynomial 
order of the trial functions $(\varphi \circ {\xi_T})$ is $r$ independent of the interface approximation.

We consider a $C^{3}$- parameterized interface $\Gamma$, which is not matched by the triangulation ${\cal T}_h$. The triangulation induces a discrete interface $\Gamma_h$, which 
is a quadratic (and in max.$\,n_{l}$ elements a linear) approximation to $\Gamma$. The discrete interface splits the triangulation in subdomains $\Omega_h^1$ and 
$\Omega_h^2$, such that each {subcell $T\in {\cal T}_h$} is either completely included in $\Omega_h^1$ or in $\Omega_h^2$. 

We consider the following discrete variational formulation: \emph{Find $u_h\in V_h$ such that}
\begin{align}\label{DiscVar}
  a_h(u_h, \phi_h) = (f_h, \phi_h)_{\Omega} \quad \forall \phi_h \in V_h,  
\end{align}
where we set $f_h|_{\Omega_h^i} := f_i, i=1,2$ and $f_i$ is a smooth extension of $f|_{\Omega_i}$ to $\Omega_h^i$. The bilinear form is given by 
\begin{align*}
  a_h(u_h,\phi_h): = (\nu_h \nabla u_h,\nabla \phi_h)_{\Omega},
\end{align*}
where $\nu_h$ is defined by
\begin{align*}
  \nu_h = \begin{cases}
    \nu_1, & \vec{x}\in \Omega_h^1\\
    \nu_2, & \vec{x}\in \Omega_h^2.
  \end{cases} 
\end{align*}

{
\begin{remark}\label{rem.3d}
The locally modified finite element method has straight-forward extensions to 3 space dimensions. One possibility is to use a hexahedral patch mesh, where each hexahedron is subdivided into 8 sub-hexahedra. The hexahedra affected by the interface are then further subdivided into 6 tetrahedra to resolve the interface. 4 different types of cuts have to be considered, based on the number of patch vertices that remain on each side of the interface (1 vs.\,7, 2 vs.\,6, 3 vs.\,5 or 4 vs.\,4). In order to guarantee a maximum angle condition in pathological situations some of the patch vertices can be moved, if necessary. Such an approach has been implemented by Langer \& Yang, see~\cite{LangerYang}. Alternatively, the patches can also be subdivided into polyhedral sub-elements as in H\"ollbacher \& Wittum~\cite{HoellbacherWittum, Vogeletal2013}. For all variants the ideas as well as the analysis presented in this paper have a straight-forward extension to the three-dimensional method.
\end{remark}}

\section{A priori error analysis}
\label{sec.apriori}

Let $h_P$ be the maximum size of a patch element $P\in {\cal T}_{P}$ of the regular patch grid. We will denote the mismatch between $\Omega_h^i$ 
and $\Omega^i$ by $S_h^i$, $i=1,2$ (see Figure \ref{fig:S_h})
\begin{align*}
  S_h^1 &:= \Omega_h^1 \setminus \Omega_1 = \Omega_2 \setminus \Omega_h^2,\\
  S_h^2 &:= \Omega_h^2 \setminus \Omega_2 = \Omega_1 \setminus \Omega_h^1.
\end{align*}
Moreover, we denote the set of elements $T\in {\cal T}_h$ that contain parts of $S_h^i$ by 
\begin{align*}
  S_T^i &:= \{T\in {\cal T}_h\, | \, T \cap S_h^i \neq 0\},\qquad  S_T:= S_T^1 \cup S_T^2.
\end{align*}
Further, we split $S_h^i$ into parts $S_{h,{\text{lin}}}^i$ with a linear approximation of the interface and parts 
${S_{h,{\text{qu}}}^i}$ with a quadratic approximation.  
Finally, by a slight abuse of notation, we will use the same notation, e.g. $S_T, S_T^i$,
for the region that is spanned by the union of all {elements} in these sets.

{By constants $c$ we will denote in the following generic constants that are independent of the mesh size $h_P$, the position of the interface, the solution $u$ and the number of linearly approximated elements $n_l$ {(but may depend on the parameters $\nu_i$, $i=1,2$).} We note that within a sequence of inequalities, $c$ might even represent different values on different sides of the inequalities.}
\begin{figure}[t]
    \centering
  \includegraphics[trim=50mm 130mm 50mm 40mm,clip,width=0.45\textwidth]{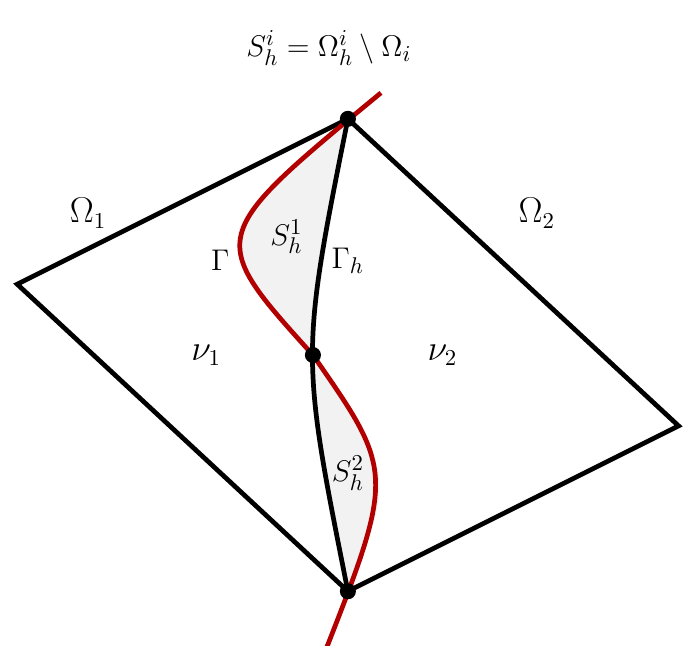}
  \caption{Mismatch between $\Omega^i$ and $\Omega_h^i$, $i=1,2$ at
    two elements along the curved interface. }
  \label{fig:S_h}
\end{figure}
\subsection{Auxiliary estimates}

We begin with some technical estimates that will be needed in order to control the mismatch between continuous and discrete 
bilinear forms. {To this purpose we will need the following Sobolev imbedding for $2\leq p<\infty$
\begin{align}\label{Sobolev}
\|u\|_{L^p(\Omega)} \leq c p^{\frac{1}{2}} \|u\|_{H^1(\Omega)},
\end{align}
which is valid with a constant $c$ independent of $p$, see~\cite{Tanakaetal2015}. We will need the following technical result

\begin{lemma}\label{lem.Jp}
Let $\alpha\in \mathbb{N}, \,h\in \mathbb{R}_+$ and $J(p) := h^{-\frac{\alpha}{p}} p^{\frac{1}{2}}$. It holds that
\begin{align*}
\min_{p\in [2,\infty]} J(p) \leq c |\ln(h)|^{\frac{1}{2}}.
\end{align*}
\end{lemma}
\begin{proof}
The necessary condition for a local minimum is 
\begin{align*}
J'(p) &= h^{-\frac{\alpha}{p}} \ln(h) \frac{\alpha}{p^2} p^{\frac{1}{2}} +\frac{1}{2} h^{-\frac{\alpha}{p}} p^{-\frac{1}{2}} 
= h^{-\frac{\alpha}{p}} p^{-\frac{1}{2}} \left( \ln(h) \frac{\alpha}{p} +\frac{1}{2} \right) \overset{!}{=}0,
\end{align*}
which yields $p=-2\alpha \ln(h)$. The minimum value is 
\begin{align*}
J(-2\alpha \ln(h)) = -e^{-\frac{1}{2}} \sqrt{2\alpha} \ln(h)^{1/2}.
\end{align*}
The fact that $\lim\limits_{p\to \infty}J(p)=\infty$ and $J(2) > J(-2\alpha \ln(h))$ show that the local minimum is in fact a global one.
\end{proof}

\noindent The following lemma will be needed to estimate the mismatch between continuous and discrete bilinear forms.}

\begin{lemma}[Geometry Approximation]\label{lem.Sh}
  Let $T\in S_T$ and let $s$ be the local approximation order of the interface, i.e. 
  \begin{equation}\label{dist}
    \text{dist }(\Gamma_h\cap T; \Gamma\cap T) \leq ch_P^{s+1}.
  \end{equation}
  If the number of elements with a linear interface approximation is bounded by $n_l$, it holds for the areas of the regions $S_{h,lin}$ and $S_{h,qu}$ that
  \begin{align}
    |S_{h,lin}| \leq n_l h_P^3, \quad |S_{h,qu}| \leq h_P^3.\label{Shest}
  \end{align}
  For $u\in H^1(\Omega_1\cup\Omega_2)$ and $\phi_h\in {\cal V}_h$ we have the bounds
  \begin{align}
    \|\nabla \phi_h\|_{S_h\cap T} &\leq ch_P^{\frac{s}{2}}\|\nabla \phi_h\|_{T}\label{DiscShi}\\ 
    \| u\|_{S_h\cap T} &\leq ch_P^{\frac{s+1}{2}} \|u\|_{\Gamma\cap T} + ch_P^{s+1} \|\nabla u\|_{S_h \cap T}. \label{L2u}
  \end{align}
  Moreover, we have for $u\in H^1(\Omega_1\cup\Omega_2)$ and $v\in H^2(\Omega_1\cup\Omega_2)$
  \begin{align}
    \|u\|_{S_{h,\text{lin}}} \leq ch_P\|u\|_{H^1(\Omega_1\cup\Omega_2)},\qquad
    \|u\|_{S_{h,\text{qu}}} \leq ch_P^{\frac{3}{2}}\|u\|_{H^1(\Omega_1\cup\Omega_2)}\label{fullqu}
  \end{align} 
  and
  \begin{align}
    {\|u\|_{S_{h,\text{lin}}} \leq c n_l^{\frac{1}{2}} h_P^{\frac{3}{2}} |\ln(h)|^{1/2} \|u\|_{H^1(\Omega_1\cup\Omega_2)}},
    \qquad  \|v\|_{S_{h,\text{lin}}} \leq cn_l^{\frac{1}{2}} h_P^{\frac{3}{2}}\|v\|_{H^2(\Omega_1\cup\Omega_2)}.   
    \label{hp32eps}
  \end{align}
  For functions $u\in H^1_0(\Omega)$ the $H^1$-norm on the right-hand side of \eqref{fullqu} and \eqref{hp32eps} can be replaced by the $H^1$-seminorm.      
\end{lemma}
\begin{proof}
  Estimates {\eqref{dist},\,\eqref{DiscShi} and \eqref{L2u}} have been shown in~{\cite[Lemmas 4.32 and 4.34]{RichterBook}}\footnote{{The proof of Lemma 4.34 in~\cite{RichterBook} contains a typo. The Poincar\'e-type inequality is there given as $\|u_h\|^2_{S_h\cap T} \le c h_P^{s+1} \|u\|_{\Gamma_h\cap T}^2 + ch_P^{2s+1} \|\nabla u\|_{S_h\cap T}^2$ with the non-optimal order $2s+1$ instead of $2s+2$. The difference comes from a transmission error in the line above (4.17) in~\cite{RichterBook}. Integration of the inequality leading to (4.17) brings the factor $h_P^{s+1}$ in addition to the factor $h_P^{s+1}$, which is already present by estimation of the distance between $\Gamma$ and $\Gamma_h$. The statement of Lemma 4.34 in~\cite{RichterBook} is indeed correct and curing the proof is trivial by just correcting the typo}.} \eqref{Shest} follows from \eqref{dist} and simple geometric arguments.
  For \eqref{DiscShi} and \eqref{L2u} a Poincar\'e-type estimate is used, see \cite[Lemma 4.34]{RichterBook}
  \begin{align}\label{localest}
    \|u\|_{S_h\cap T}^2 \leq ch_P^{s+1} \|u\|_{\Gamma_h\cap T}^2 + ch_P^{2s+2} \|\nabla u\|_{S_h \cap T}^2.  
  \end{align} 
  Summation over all elements in $S_{T,\text{lin}}$ and $S_{T,\text{qu}}$, respectively, and a global trace inequality for the interface terms yields \eqref{fullqu}.
  To show \eqref{hp32eps}, we use a H\"older inequality for $p\in [2,\infty]$
    \begin{align}\label{Lpest}
      \|u\|_{S_{h,lin}} \leq  |S_{h,lin}|^{\frac{1}{2}-\frac{1}{p}} \|u\|_{L^p(S_{h,lin})}.
    \end{align}
    {Due to $|S_{h,lin}|\leq cn_l h_P^3$ and the Sobolev imbedding \eqref{Sobolev} for $\Omega=\Omega_i$ we have for arbitrary $p\in [2,\infty)$
    \begin{align}\label{fminusfh_lin}
      \|u\|_{S_{h,lin}} \leq c p^{\frac{1}{2}} n_l^{\frac{1}{2}} h_P^{\frac{3}{2}-\frac{3}{p}} \|u\|_{H^1(\Omega_1\cup\Omega_2)}.
    \end{align}
    Using Lemma~\ref{lem.Jp}, we obtain the first estimate in \eqref{hp32eps}.
} 
    If $v\in H^2(\Omega_1\cup\Omega_2)$ we can use \eqref{Lpest} for 
    $p=\infty$ due to the Sobolev imbedding $H^2(\Omega_i) \subset L^\infty(\Omega_i)$ and we obtain
    \begin{align}
      \|u\|_{S_{h,lin}} \leq cn_l^{\frac{1}{2}} h_P^{\frac{3}{2}} \|u\|_{H^2(\Omega_1\cup\Omega_2)}. 
    \end{align}

    Finally, the norms on the right-hand side can be substituted by the $H^1$-seminorm for $u\in H^1_0(\Omega)$ by means of the Poincar\'e inequality.

\end{proof}

In the following estimates the mismatch between discrete and continuous bilinear form will be the predominant issue and will lead to some technicalities in the estimates. The continuous solution $u$ is regular in $\Omega_1$ and $\Omega_2$, while its normal derivative has a jump across $\Gamma$. Discrete functions can only have irregularities at the boundaries of cells $\partial T$, which means that a discrete function can {only} resemble a similar discontinuity across the discrete interface $\Gamma_h$.

To cope with this difference, we will need a map $\pi:  H^3(\Omega_1\cup\Omega_2) \to H^3(\Omega_h^1\cup\Omega_h^2)$. To define the map, let $u\in H^3(\Omega_1\cup\Omega_2)$ and $u_i:=u|_{\Omega_i}\in H^3(\Omega_i)$ the restriction to the subdomain $\Omega_i, i=1,2$. We use smooth extensions $\tilde u_i\in H^3(\Omega)$ (i=1,2) to the full domain $\Omega$. Such an extension {exists} with the properties
  \begin{equation}\label{extension}
{\tilde{u}_i = u \quad \text{in }\Omega_i,} \qquad    
     \|\tilde u_i\|_{H^m(\Omega)}\le C
    \|u\|_{H^m(\Omega_i)}, \quad i=1,2, \quad m=2,3, 
  \end{equation}
 as the interface $\Gamma$ is smooth,
  see e.g.\,the textbook of Stein~\cite{Stein}[Section VI.3.1]. We use these extensions to define a function $\pi u \in H^3(\Omega_h^1\cup\Omega_h^2)$:
  \begin{align}\label{piu}
  \pi u{(\vec{x})} = \begin{cases}
    \tilde{u}_1{(\vec{x})},& \vec{x} \in \Omega_h^1,\\
    \tilde{u}_2{(\vec{x})},& \vec{x} \in \Omega_h^2.
  \end{cases}
  \end{align}
It should be noted that $\pi u$ can be discontinuous across $\Gamma_h$.  
  
The following estimate analyzes the difference between $u$ and $\pi u$ in the $H^1$-seminorm.
\begin{lemma}\label{lem.tildeu}
Let $u\in H^3(\Omega_1\cup\Omega_2)$, $\pi u
\in H^3(\Omega_h^1\cup\Omega_h^2)
$ the function defined by \eqref{piu}
and $n_l$ the maximum number of elements with a linear interface approximation. It holds that 
\begin{align}
\|\nabla (u-\pi u) \|_{\Omega} \leq ch_P \left(n_l^{1/2} +1\right) \|u\|_{H^2(\Omega_1\cup\Omega_2)}\label{upiu2}\\
\|\nabla (u-\pi u) \|_{\Omega} \leq ch_P^{3/2} \left(n_l^{1/2} +1\right) \|u\|_{H^3(\Omega_1\cup\Omega_2)}\label{upiu}. 
\end{align}
\end{lemma}  
\begin{proof}
$u$ and $\pi u$ differ only in the small strip $S_h$ around the interface. For $u\in H^3(\Omega_i)$ we have, using the Sobolev embedding $H^3(\Omega_i)\subset W^{1,\infty}(\Omega_i)$ and the continuity of the extensions \eqref{extension}
\begin{align*}
\|\nabla (u-\pi u) \|_\Omega = \|\nabla (u-\pi u) \|_{S_h} &\leq |S_h|^{\frac{1}{2}} \left(\|\nabla u\|_{L^{\infty}(\Omega)} +\|\nabla \pi u\|_{L^{\infty}(\Omega)}\right) \\
&\leq c |S_h|^{\frac{1}{2}} \|u\|_{H^3(\Omega_1\cup\Omega_2)}.
\end{align*}
\eqref{upiu} follows by means of \eqref{Shest}.

To show \eqref{upiu2}, we note that $u-\pi u$ vanishes in cells $T\in {\cal T}_h \setminus S_T$. Thus, let $T\in S_T$ and let $s\in\{1,2\}$ be the local approximation order of the interface in $T$. We use \eqref{L2u} and the fact that $s\geq 1$ to get
\begin{align*}
\|\nabla (u-\pi u) \|_T = \|\nabla &(u-\pi u) \|_{S_h\cap T}
\leq ch_P^{\frac{1+s}{2}} \|\nabla (u-\pi u)\|_{\Gamma\cap T} + ch_P^{1+s} \|\nabla^2(u-\pi u)\|_{S_h \cap T}\\
&\leq ch_P \left(\|\nabla u\|_{\Gamma\cap T} + \|\nabla \pi u\|_{\Gamma\cap T}\right)  + ch_P^2 \left( \|\nabla^2 u\|_{S_h \cap T} + \|\nabla^2 \pi u\|_{S_h \cap T}\right),
\end{align*}
where the derivatives on $\Gamma$ need to be seen from $S_h$.

After summation over all cells $T\in {\cal T}_h$ a global trace inequality and the continuity of the extension \eqref{extension} yield
\begin{align*}
\|\nabla (u-\pi u) \|_\Omega \leq ch_P \left( \|u\|_{H^2(\Omega_1\cup\Omega_2)} +\|\pi u\|_{H^2(\Omega_1\cup\Omega_2)}\right) \leq ch_P \|u\|_{H^2(\Omega_1\cup\Omega_2)}.
\end{align*}

\end{proof}


\subsection{Interpolation}

In this subsection, we will derive interpolation estimates for a Lagrangian interpolant $I_h$. {Let ${\cal L}_T$ be the set of Lagrange points that belong to a cell $T\in {\cal T}_h$.}
In the case of a linear interface approximation, it can happen that some of these lie on $\Gamma_h$, but not on $\Gamma$.
This means that there are elements with Lagrange points $\vec{x}_i\in{\cal L}_T$, that lie in different subdomains $\Omega_1$ and $\Omega_2$, see Figure~\ref{fig:interpolation}. Defining the interpolant as 
$I_h u = \sum_{i\in {\cal L}_T} u(\vec{x}_i)$ would lead to a poor approximation order (${\cal O}(h_P)$ in the $H^1$-norm), due to the discontinuity of $\nabla u$ across $\Gamma$. Each such point $\vec{x}_i$ lies, however, on a line between two points $\vec{x}_1^*$ and $\vec{x}_2^*$ on $\Gamma$. We use a linear interpolation of the values $u(\vec{x}_1^*)$ and $u(\vec{x}_2^*)$ in order to define $I_h u(\vec{x}_i) \coloneqq \frac{1}{2}(u(\vec{x}_1^*)+u(\vec{x}_2^*))$, see also {\cite[Eqn. (6.13)]{FreiDiss}} and Fig.~\ref{fig:interpolation}. 
%

\begin{figure}[t]
  \begin{center}
    \includegraphics[width=0.8\textwidth]{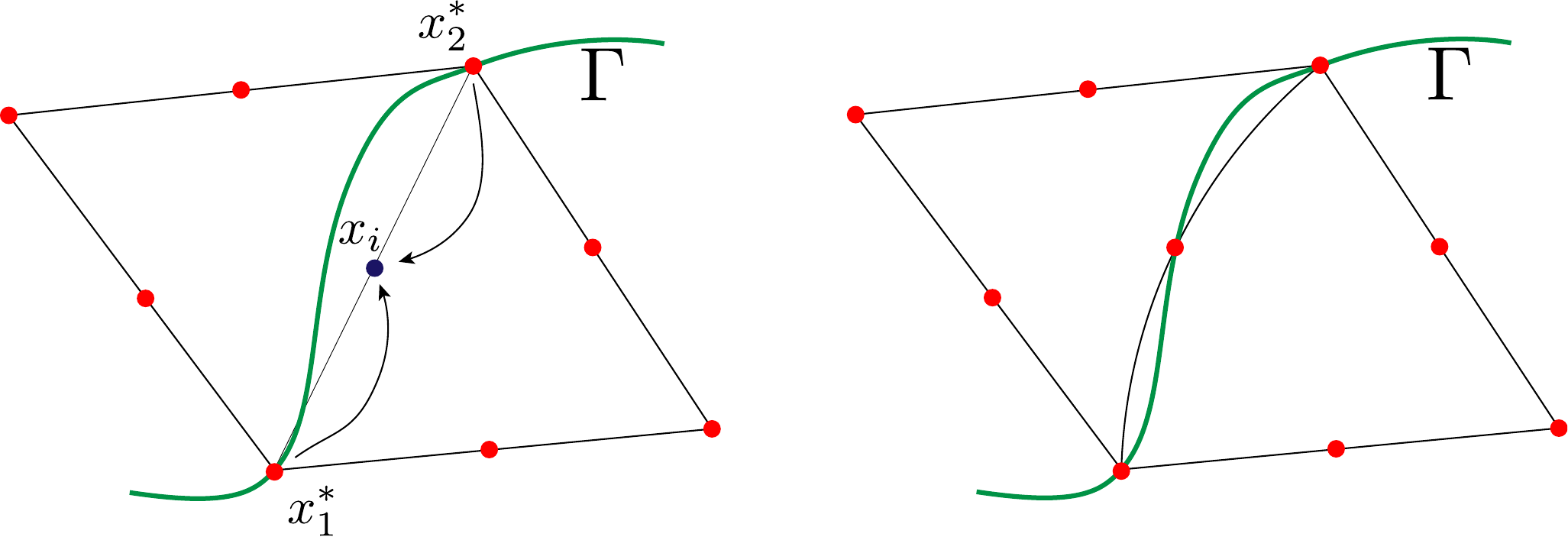}
  \end{center}
  \caption{Interpolation operator at the interface. If the interface is approximated with second order (right) we use the standard nodal interpolation. For linear interface approximations (left) we replace the node in the middle of the interface edge by the mean of the two adjacend corner nodes.}
  \label{fig:interpolation}
\end{figure}
We have the following approximation properties for this modified Lagrangian interpolant.

\begin{lemma}[Interpolation]\label{lem.interpol}
  Let $u\in {\cal U}:=\big[H_0^1(\Omega)\cap H^3(\Omega_1 \cup \Omega_2)]$ and $\tilde u = \pi u \in H^3(\Omega_h^1\cup\Omega_h^2)$ the function resulting from the map $\pi$ defined in \eqref{piu}. Moreover, we assume that $\Gamma$ is a smooth interface with $C^3$-parametrization and that the interface is approximated with second order in all elements $T\in {\cal T}_h$, except for maximum $n_{l}$ elements, 
  where the interface approximation is linear. It holds for the Lagrangian interpolation operator 
  $I_h: {\cal U} \to V_h$ that
  \begin{align}
    \|\nabla^m (u-I_h u)\|_{\Omega}&\leq c h_P^{2-m} \|u\|_{H^2(\Omega_1\cup\Omega_2)},\quad m=1,2\label{interpol1} \\
    \|\nabla (\tilde{u}- I_h u)\|_{\Omega}&{ \leq \left(c_l n_l^{1/2} |\ln(h)|^{1/2} + c_q\right)  h_P^2 \|u\|_{H^{3}(\Omega_1\cup\Omega_2)}.} \label{interpoltilde}
  \end{align}
{where $c_l$ and $c_q$ are generic constants that correspond to patches with a linear and a quadratic interface approximation, respectively.}  
 {For $u\in W^{2,\infty}(\Omega_1\cup \Omega_2)$ we have further
 \begin{align}\label{interpolW2infty}
 \|\nabla (\tilde{u}- I_h u)\|_{\Omega}& \leq \left( c_l n_l^{1/2} + c_q \right) h_P^2  \|u\|_{W^{2,\infty}(\Omega_1\cup\Omega_2)}.
 \end{align}} 
\end{lemma}
\begin{proof}
  First, we note that it holds $I_h u = I_h\tilde{u}$ by construction of the interpolant, as $u(\vec{x}_i)=\tilde{u}(\vec{x}_i)$ in all Lagrange points, that are used in the definition of the interpolant $I_h$.
  
Next, we note that the proof of all estimates is standard in all elements $T\in{\cal T}_h\setminus S_T$ that are not affected by the interface, since $\tilde{u}|_T = u|_T, u\in H^3(T)$  and
  \begin{align}\label{standardint}
    \|\nabla^m (u-I_h u)\|_{T}&\leq ch_P^{2-m} \|u\|_{H^2(T)} \, (m=1,2), \qquad \| \nabla (u-I_h u)\|_{T}\leq ch_P^2 \|u\|_{H^3(T)}.
  \end{align}
In elements $T\in S_{T,qu}$ there are no Lagrange points on $\Gamma_h\setminus\Gamma$ and $I_h$ is the standard Lagrangian interpolant. As $\tilde{u}$ is smooth in $T$, estimate \eqref{interpoltilde} 
is also standard 
\begin{align*}
\|\nabla (\tilde{u}- I_h u)\|_T &= 
\|\nabla (\tilde{u}- I_h \tilde{u})\|_T
\leq ch_P^2 \|\tilde{u}\|_{H^3(T)}. 
\end{align*}
In elements $T\in S_{T,lin}$ the interpolation is only linear due to the modification described above. Let $T\in S_{T,lin}^i$ with $i\in\{1,2\}$ and let $P$ be the patch that contains $T$. The following estimate has been shown {in~\cite[Lemma 6.14]{FreiDiss}}
\begin{align}\label{shlin}
\|\nabla (\tilde{u}- I_h u)\|_T &= 
\|\nabla (\tilde{u}- I_h \tilde{u})\|_T
\leq ch_P \|\nabla^2 \tilde{u}_i\|_P,
\end{align} 
where $\tilde{u}_i$ denotes the extension of $u_i$ to $\Omega$. 
We sum over all elements $T\in S_{T,lin}^i$ and denote by $S_{P,lin}^i$ the region, which is spanned by the patches $P$ containing elements $T\in S_{T,lin}^i$. It holds {with} $|S_{P,lin}^i|=n_l{\cal O}(h_P^2)$
\begin{align*}
  \|\nabla (\tilde{u}- I_h u)\|_{S_{T,lin}^i} &\leq 
  ch_P \|\nabla^2 \tilde{u}_i\|_{S_{P,lin}^{i}}\\
  &\leq ch_P |S_{P,lin}^i|^{\frac{1}{2} -\frac{1}{p}} \|\tilde{u}_i\|_{W^{2,p}(\Omega_i)}
  {\leq c n_l h_P^{2-\frac{2}{p}} \|\tilde{u}_i\|_{W^{2,p}(\Omega_i)}}.
\end{align*}
{The estimate \eqref{interpolW2infty} follows for $p=\infty$.
 To show \eqref{interpoltilde}, we can estimate further by using the Sobolev imbedding \eqref{Sobolev} for $p<\infty$
 \begin{align*}
  \|\nabla (\tilde{u}- I_h u)\|_{S_{T,lin}^i}  \,\leq\, c p^{1/2} n_l^{1/2} h_P^{2-2/p} \|u\|_{H^3(\Omega_i)}.
 \end{align*}
The estimate \eqref{interpoltilde} follows by means of Lemma~\ref{lem.Jp}.}
To show \eqref{interpol1} 
we split into
  \begin{align}
    \|\nabla^m (u- I_h u)\|_{T}&\le  \|\nabla^m (u - \tilde u)\|_{T} +
    \|\nabla^m (\tilde u - I_h  \tilde u)\|_{T}.\label{X1}
  \end{align}
   For $m=1$, the first term has been estimated in Lemma~\ref{lem.tildeu},
    for $m=2$ we can use the stability of the extension \eqref{extension}. 
    \eqref{interpol1} follows from \eqref{shlin} {in $S_{T,lin}$} and standard interpolation estimates {elsewhere}.
\end{proof}


\subsection{A priori error estimate}

We are now ready to prove the main result of this section. {To this end, we introduce the discrete energy norm
\begin{align*}
\vertiii{u-u_h} := \left(\|{\nu_1^{1/2}} \nabla (\tilde{u}_1 - u_{h}^1)\|_{\Omega_h^1}^2 + \|{\nu_2^{1/2}} \nabla (\tilde{u}_2 - u_{h}^2)\|_{\Omega_h^2}^2\right)^{1/2},
\end{align*}
where $\tilde{u}_i$ are smooth extensions of $u_i = u|_{\Omega_i}$ to $\Omega_h^i$ and $u_h^i := u_{h}|_{\Omega_h^i}.$}

\begin{theorem}[A priori estimate]
  \label{t10}
  Let $\Omega \subset \mathbb{R}^2$ be a convex domain with polygonal boundary, which is resolved (exactly) by the {family of} triangulations ${\cal T}_h$. 
  We assume 
  a splitting
  $\Omega=\Omega_1\cup \Gamma \cup \Omega_2$, where $\Gamma$ is a smooth interface with $C^3$-parametrization and that the solution $u$ to \eqref{VarForm} belongs to $H^3(\Omega_1\cup\Omega_2)$.
  Moreover, we denote
  by $n_l$ the maximum number of elements {$T\in {\cal T}_h$}, where the interface is approximated linearly. 
  For the locally modified finite element solution $u_h\in V_h$ to \eqref{DiscVar} it holds
  \begin{align}
    {\vertiii{u- u_h}}         
    &{\leq  \Big(c_l n_l^{\frac{1}{2}} |\ln(h)|^{1/2} +c_q\Big) h_P^2  \|u\|_{H^3(\Omega_1\cup\Omega_2)},}\label{energytilde}\\
    \|u-u_h\|_{\Omega} &\leq {\Big(c_l n_l |\ln(h)|^{1/2} + c_q \Big) h_P^3 \|u\|_{H^3(\Omega_1\cup\Omega_2)}}.\label{L2result}
  \end{align}
  {where $c_l$ and $c_q$ are generic constants that correspond to patches with a linear and a quadratic interface approximation, respectively.}  
  {For $u\in W^{2,\infty}(\Omega_1\cup \Omega_2)$ we have further
  \begin{align}
     \vertiii{u- u_h}         
    \leq \Big(c_l n_l^{\frac{1}{2}} +c_q\Big) h_P^2 \left( \|u\|_{H^3(\Omega_1\cup\Omega_2)} + \|u\|_{W^{2,\infty}(\Omega_1 \cup\Omega_2)}\right).\label{energyW2infty}
    \end{align}}
 \end{theorem}


%
\begin{proof}
  \emph{(i)} First, we have the following perturbed Galerkin orthogonality by subtracting \eqref{DiscVar} from \eqref{VarForm}
  \begin{align}\label{perturbed}
    a(u,\phi_h) - a_h(u_h, \phi_h) &= (f-f_h,\phi_h)_{\Omega} \quad \forall \phi_h\in V_h.
  \end{align}
  We start by estimating the right-hand side in \eqref{perturbed}. The difference $f-f_h$ vanishes everywhere besides on $S_h$. We have
  \begin{align*}
    (f-f_h,\phi_h)_{\Omega} = (f-f_h,\phi_h)_{S_h} \leq \left(\|f_1\|_{S_h} +\|f_2\|_{S_h}\right) \|\phi_h\|_{S_h},
  \end{align*}
  where $f_i$ denotes a smooth extension of $f|_{\Omega_i}$ to $\Omega, i=1,2$.

  We split the region $S_h$ into parts with a quadratic interface approximation $S_{h,qu}$ and parts with a linear approximation. \eqref{fullqu} and \eqref{hp32eps} yield
  \begin{equation}\label{fminusfh_qu}
    \begin{aligned}
      \|f_i\|_{S_{h,qu}}+ \|f_i\|_{S_{h,lin}} 
      &\leq c h_P\|f\|_{H^1(\Omega_1\cup\Omega_2)} {\,\leq\, c h_P\|u\|_{H^3(\Omega_1\cup\Omega_2)}} \quad\text{ and } \\
      \|f_i\|_{S_{h,qu}}+ \|f_i\|_{S_{h,lin}} 
      &\leq \Big(c h_P^{\frac{3}{2}} + {c n_l^\frac{1}{2} h_P^{\frac{3}{2}}|\ln(h)|^{1/2}\Big) \|f\|_{H^1(\Omega_1\cup\Omega_2)}}\\
      &{\leq \Big(c h_P^{\frac{3}{2}} + c n_l^\frac{1}{2} h_P^{\frac{3}{2}}|\ln(h)|^{1/2}\Big) \|u\|_{H^3(\Omega_1\cup\Omega_2)}.}      
    \end{aligned}
  \end{equation}
  The second estimate yields
  \begin{align}\label{fminusfh}
    (f-f_h,\phi_h)_{\Omega} \leq {ch_P^{\frac{3}{2}} \left(1+ n_l^{\frac{1}{2}}|\ln(h)|^{1/2}\right) \|u\|_{H^3(\Omega_1\cup\Omega_2)}} \|\phi_h\|_{S_h}.
  \end{align}
  \emph{(ii)}  
  For the energy norm estimate 
  we start by splitting into interpolatory and discrete part
  \begin{equation}\label{startenergy}
    \begin{aligned}
      {\vertiii{u-u_h}}
      \leq \|\nu_h^{1/2} \nabla(\tilde{u} - I_h u)\|_{\Omega} + \|\nu_h^{1/2} \nabla (I_h u - u_h)\|_{\Omega}.
    \end{aligned}
  \end{equation}
 The interpolatory part has already been estimated in Lemma~\ref{lem.interpol}. For the second term in \eqref{startenergy}, we use the perturbed Galerkin orthogonality (\ref{perturbed}) with $\varphi_h:=I_h u-u_h$ 
  \begin{align}
    \|\nu_h^{1/2} \nabla (I_h u - u_h)\|_{\Omega}^2
    &= \left(\nu_h \nabla (I_h u - u_h), \nabla (I_h u-u_h)\right)_{\Omega}\notag\\
    &=  \left(\nu_h \nabla I_h u - \nu \nabla u, \nabla (I_h u-u_h)\right)_{\Omega} + (f-f_h, I_h u -u_h)_{\Omega}.\label{splitffh}
  \end{align}
  We split the first part in \eqref{splitffh} further
  \begin{align}\label{split27}
  \begin{split}
    \left(\nu_h \nabla I_h u - \nu \nabla u, \nabla (I_h u-u_h)\right)_{\Omega}\,
= \, &\left(\nu_h \nabla (I_h u - \tilde{u}), \nabla (I_h u-u_h)\right)_{\Omega}\\
&\qquad+ \left(\nu_h \nabla \tilde{u} - \nu \nabla u, \nabla (I_h u-u_h)\right)_{\Omega}.    
\end{split}
\end{align}
 For the first part, we use \eqref{interpoltilde} to get
 \begin{align}\label{split28}
   \left(\nu_h \nabla (I_h u - \tilde{u}), \nabla (I_h u-u_h)\right)_{\Omega}
  {\leq\, ch_P^2 \left(n_l^{\frac{1}{2}} |\ln(h)|^{1/2} +1\right) \|u\|_{H^3(\Omega_1\cup\Omega_2)}} \|\nu_h^{1/2} \nabla (I_h u-u_h) \|_\Omega. 
 \end{align}
  The integrand in the second term on the right-hand side of \eqref{split27} vanishes everywhere besides on $S_h$. We obtain {by} the Sobolev imbedding $H^3(\Omega_i)\subset W^{1,\infty}(\Omega_i)$, the continuity of the extension \eqref{extension} and \eqref{DiscShi} and \eqref{Shest} from Lemma~\ref{lem.Sh}
  \begin{align}
  \begin{split}\label{nuhtildeunuu}
   \left(\nu_h \nabla \tilde{u} - \nu \nabla u, \nabla (I_h u-u_h)\right)_{\Omega} 
&=   \left(\nu_h \nabla \tilde{u} - \nu \nabla u, \nabla (I_h u-u_h)\right)_{S_h}\\
   &\leq c \left(\left\|\nabla \tilde{u}\right\|_{S_h} + \left\| \nabla u\right\|_{S_h}\right) \|\nu_h^{1/2} \nabla (I_h u-u_h)\|_{S_h}\\
   &\leq c |S_h|^{1/2} \|u\|_{W^{1,\infty}(\Omega_1\cup\Omega_2)} h_P^{1/2}
   \|\nu_h^{1/2} \nabla (I_h u-u_h)\|_{S_T} \\
   &\leq ch_P^2 \left(n_l^{1/2} +1\right) \|u\|_{H^3(\Omega_1\cup\Omega_2)}
   \|\nu_h^{1/2} \nabla (I_h u-u_h)\|_{\Omega}
   \end{split}
  \end{align}
  For the second term in \eqref{splitffh}, we use \eqref{fminusfh_qu} and \eqref{fullqu}
  \begin{align*}
    (f-f_h, I_h u -u_h)_{\Omega} &\leq c h_P \|u\|_{H^3(\Omega_1\cup\Omega_2)} \|I_h u -u_h\|_{S_h}\\
&\leq c h_P^2 \|u\|_{H^3(\Omega_1\cup\Omega_2)} \|\nu_h^{1/2}\nabla (I_h u - u_h)\|_{\Omega_1\cup\Omega_2}.
  \end{align*}
Combining the estimates, we obtain 
  \begin{align*}
    \big\|\nu_h^{\frac{1}{2}} \nabla (I_h u - u_h)\big\|_{\Omega_1\cup\Omega_2} \leq ch_P^2 \|u\|_{H^3(\Omega_1 \cup\Omega_2)}.
  \end{align*}
  This completes the proof of \eqref{energytilde}. {The proof of \eqref{energyW2infty} follows exactly the same lines, with the only difference that we use \eqref{interpolW2infty} instead of \eqref{interpoltilde} in \eqref{split28} to get
   \begin{align}
   \left(\nu_h \nabla (I_h u - \tilde{u}), \nabla (I_h u-u_h)\right)_{\Omega}
   \leq\, ch_P^2 \left(n_l^{\frac{1}{2}} +1\right) \|u\|_{W^{2,\infty}(\Omega_1\cup\Omega_2)} \|\nu_h^{1/2} \nabla (I_h u-u_h) \|_\Omega. 
 \end{align}  
  }

  \medskip\noindent\emph{(iii)} To estimate the $L^2$ - norm error, we define the following adjoint problem. Let $z\in H_0^1(\Omega)$ be the solution of 
  \[
  (\nu\nabla \varphi , \nabla z)=\|e_h\|^{-1} (e_h, \varphi)_\Omega \quad \forall \varphi \in H_0^1(\Omega).
  \]
  The solution $z$ lies in in $H_0^1(\Omega)\cap H^2(\Omega_1 \cup \Omega_2)$ and satisfies
  \[
  \begin{aligned}
    \| z\|_{H^2(\Omega_1 \cup \Omega_2)} \leq c_s.
  \end{aligned}
  \]
  By choosing $\varphi = u-u_h=e_h$ and adding and subtracting $\nu_h\nabla u_h$, we have
  \begin{align}\label{L2split}
    \|e_h\|=(\nu\nabla e_h , \nabla z)_\Omega
    = (\nu\nabla u - \nu_h \nabla u_h, \nabla z)_\Omega + ((\nu_h-\nu) \nabla u_h, \nabla z)_\Omega.
  \end{align} 
  For the second term in \eqref{L2split}, we have 
  \begin{align}\label{L2split2}
    \begin{split}
      ((\nu_h-\nu) \nabla u_h, \nabla z)_\Omega &= ((\nu_h-\nu) \nabla u_h, \nabla z)_{S_h}
      \leq C \left( \|\nu_h \nabla u_h \|_{S_{h}}  \|\nabla z\|_{S_{h}} \right)
    \end{split}
  \end{align}
  We split the first term on the right-hand side further and use the bound for the energy norm error as well as \eqref{hp32eps} (Lemma~\ref{lem.Sh})
  \begin{align*}
    \| \nu_h\nabla u_h\|_{S_{h}} &\leq 
    \| \nu_h\nabla (u_h- u)\|_{S_{h}}
    + \|\nu \nabla u\|_{S_{h}}\\
    &\leq c \left(n_l^{\frac{1}{2}}+1\right) h_P^{3/2} \|u\|_{H^3(\Omega_1\cup\Omega_2)}.
  \end{align*}
  For the last term in \eqref{L2split2}, we obtain from \eqref{fullqu} 
  and \eqref{hp32eps}
  \begin{align*}
    \|\nabla z \|_{S_{h,lin}} &{\leq ch_P^{\frac{3}{2}} \left( n_l^{\frac{1}{2}} |\ln(h)|^{1/2} +1\right)  \|z\|_{H^2(\Omega_1\cup\Omega_2)} \leq  ch_P^{\frac{3}{2}} \left( n_l^{\frac{1}{2}} |\ln(h)|^{1/2} +1\right),}\\  
    \|\nabla z \|_{S_{h,qu}} &\leq c h_P^{\frac{3}{2}}  \|z\|_{H^2(\Omega_1\cup\Omega_2)} \leq c h_P^{\frac{3}{2}}.
  \end{align*}
  {Altogether, we obtain} for the second term in \eqref{L2split}
  \begin{align}\label{nunuhuhz}
    ((\nu_h-\nu) \nabla u_h, \nabla z)_\Omega \leq ch_P^3 \left(n_l |\ln(h)|^{1/2} + 1 \right) \|u\|_{H^3(\Omega_1\cup\Omega_2)}.
  \end{align}
  {Concerning} the first term in \eqref{L2split}, we add and substract the interpolant $\nabla I_h z$, as well as $\pm \nu_h \tilde u$
  \begin{align}
\begin{split}  
  \label{L2split3}
    (\nu\nabla u - \nu_h \nabla u_h, \nabla z)_\Omega 
    &= (\nu\nabla u - \nu_h \nabla \tilde{u}, \nabla (z-I_h z))_\Omega
    +
    (\nu_h\nabla (\tilde{u} - u_h), \nabla (z-I_h z))_\Omega    \\
    &\qquad+ (\nu\nabla u - \nu_h \nabla u_h, \nabla I_h z)_\Omega.
    \end{split}
  \end{align}
For the first term on the right-hand side, we obtain as in \eqref{nuhtildeunuu}
  \begin{align*}
  (\nu\nabla u - \nu_h \nabla \tilde{u}, \nabla (z-I_h z))_\Omega
  &\leq ch_P^{3/2} \left(n_l^{1/2} +1\right) \|u\|_{H^3(\Omega_1\cup\Omega_2)}
   \|\nu_h^{1/2} \nabla (z-I_h z)\|_{S_h}
  \end{align*}
  {We estimate the latter norm} using \eqref{fullqu}, \eqref{hp32eps} and \eqref{interpol1} 
  \begin{align*}
    \|\nu_h^{1/2} \nabla (z-I_h z)\|_{S_h} &{\leq ch_P^{3/2} \left(n_l^{1/2}|\ln(h)|^{1/2} + 1\right) \|z-I_h z\|_{H^2(\Omega_1\cup\Omega_2)}}\\
    &{ \leq ch_P^{3/2} \left(n_l^{1/2}|\ln(h)|^{1/2} + 1\right).}
  \end{align*}
  The second term in \eqref{L2split2} is easily estimated with the bound for the energy norm and the interpolation error \eqref{interpol1} 
  \begin{align*}
  (\nu_h\nabla (\tilde{u} - u_h), \nabla (z-I_h z))_\Omega 
    &\leq c\left(n_l^{\frac{1}{2}}+1\right) h_P^3 \|u\|_{H^3(\Omega_1\cup\Omega_2)}.
  \end{align*} 
 For the third term in \eqref{L2split3}, we use the perturbed Galerkin orthogonality \eqref{perturbed}
  \begin{align}
    \begin{split}
      \label{L2split4}
      (\nu\nabla u - \nu_h \nabla u_h, \nabla I_h z)_\Omega &= 
      (f-f_h, I_h z)_{S_h} \\
      &\leq 
      \|f_1-f_2\|_{S_{h,lin}} \|I_h z\|_{S_{h,lin}} +
      \|f_1-f_2\|_{S_{h,qu}} \|I_h z\|_{S_{h,qu}}.
    \end{split}
  \end{align}
  For the first part in both terms, we use \eqref{fullqu} and \eqref{hp32eps}, respectively
  \begin{align*}
    \|f_1-f_2\|_{S_{h,lin}} +
    \|f_1-f_2\|_{S_{h,qu}}\, &\leq\,  {c h_P^{\frac{3}{2}}
    \Big(n_l^{\frac{1}{2}} |\ln(h)|^{1/2}  +
    1\Big) \|f\|_{H^1(\Omega_1\cup\Omega_2)}}\\
    & \leq\,  {c h_P^{\frac{3}{2}}
    \Big(n_l^{\frac{1}{2}} |\ln(h)|^{1/2}  +
    1\Big) \|u\|_{H^3(\Omega_1\cup\Omega_2)}.}
  \end{align*}
  For the remaining terms in \eqref{L2split4}, it is sufficient to consider the smallness of $|S_h|$, a Sobolev imbedding and the continuity of the extension~\eqref{extension}
  \begin{align*}
    \|I_h z\|_{S_{h,lin}} &\leq |S_{h,lin}|^{\frac{1}{2}} \|I_h z\|_{L^{\infty}(\Omega)} \leq cn_l^{\frac{1}{2}} h_P^{\frac{3}{2}}\|z\|_{L^{\infty}(\Omega)} \leq cn_l^{\frac{1}{2}} h_P^{\frac{3}{2}}\|z\|_{H^2(\Omega_1\cup\Omega_2)} \leq cn_l^{\frac{1}{2}} h_P^{\frac{3}{2}}\\
    \|I_h z\|_{S_{h,qu}} &\leq |S_{h,qu}|^{\frac{1}{2}} \|I_h z\|_{L^{\infty}(\Omega)} \leq ch_P^{\frac{3}{2}}.
  \end{align*}
 Altogether this yields the following estimate for the term in \eqref{L2split3}, which completes the proof of the $L^2$-norm estimate
  \begin{align*}
    (\nu\nabla u - \nu_h \nabla u_h, \nabla z)_\Omega 
    {\leq ch_P^{3} \left(n_l |\ln(h)|^{1/2}  + 1 \right) \|u\|_{H^3(\Omega_1\cup\Omega_2)}.}
  \end{align*}
\end{proof}

{
\begin{remark}{(Energy norm)}
There are different possibilities to choose the energy norm in Theorem~\ref{t10}. The result \eqref{energytilde}
could also be shown in the corresponding norm defined on the continuous subdomains $\Omega_1$ and $\Omega_2$
\begin{align}\label{norm2}
\vertiii{u-u_h}_{2} := \left(\| {\nu_1^{1/2}} \nabla (u_1 - \tilde{u}_{h}^1)\|_{\Omega_1}^2 + \|{\nu_2^{1/2}} \nabla (u_2 - \tilde{u}_{h}^2)\|_{\Omega_2}^2\right)^{1/2},
\end{align}
where $u_i = u|_{\Omega_i}$ and $\tilde{u}_h^i$ denote the canonical extensions of $u_h^i:= u_{h}|_{\Omega_h^i}$ to $\Omega_i$.
If one would consider the norm 
  \begin{align}
    \vertiii{u-u_h}_3 := \|\nu_h^{1/2}\nabla (u - u_h)\|_{\Omega}\label{tildeu}
  \end{align} 
  a reduced order of convergence, namely ${\cal O}(h_P^{\frac{3}{2}})$ would result, even for a fully quadratic interface approximation ($n_l=0$). The reason is that $\nabla u$ shows a discontinuity across $\Gamma$, while $\nabla u_h$ is discontinuous across the discrete interface $\Gamma_h$. Hence, the error in the gradient is ${\cal O}(1)$ in the strip $S_h$ between the interfaces, which is of size $|S_h|^{1/2}={\cal O}(h_P^{3/2})$. This bound is already optimal in the estimate for $\|\nabla (u-\pi u) \|_{\Omega}$ in \eqref{upiu}.
  
  We have chosen the discrete energy norm $\vertiii{u-u_h}$ in Theorem~\ref{t10}, as this is the only norm, which can be easily evaluated by numerical quadrature. A quadrature formula that evaluates the norms \eqref{norm2} or \eqref{tildeu} accurately
would need to resolve the strip $S_h$, which is non-trivial. Any standard approximation, such as a summed midpoint rule would lead to an additional quadrature error of ${\cal O}(h_P^{3/2})$, which would dominate the overall error.

\end{remark}}

{\begin{remark}{(Regularity)}
We have assumed the regularity $u\in H^3(\Omega_1\cup\Omega_2)$ (resp.\,$u\in W^{2,\infty}(\Omega_1\cup\Omega_2)$) in Theorem~\ref{t10}. This is guaranteed if both subdomains $\Omega_1$ and $\Omega_2$ are smooth (precisely $W^{3,\infty}$) and the right-hand side has regularity $f\in H^1(\Omega_1\cup\Omega_2)$ (resp.\,$f\in L^\infty(\Omega_1\cup\Omega_2)$). In this work the overall domain $\Omega$ is assumed polygonal in order to avoid additional technicalities associated with the approximation of exterior curved boundaries. For the latter we refer to the literature, for example~\cite{Bernardi1989}.
\end{remark}}



\section {Implementation}
\label{sec.impl}

\begin{figure}[bt]
  \centering
  \includegraphics[trim=50mm 70mm 10mm 40mm,clip,width=0.75\textwidth]{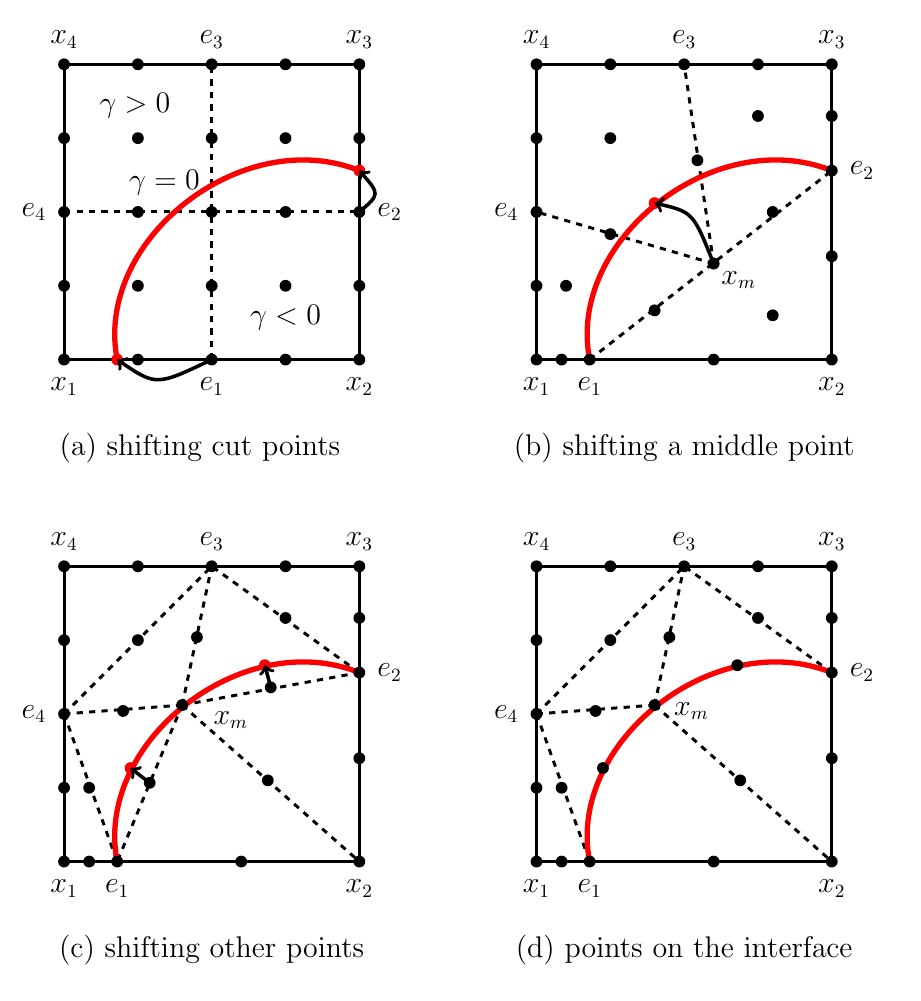}
  \caption{Rearrangement of the Lagrangian points on the interface.}
  \label{fig:implement}
\end{figure}

The locally modified finite element method is based on a patch-wise parametric approach. Let {${\cal T}_P$ be the triangulation in patches}. 
We denote by $P\in {{\cal T}_P}$ the patches, which are quadrilaterals with 25 degrees of freedom (see Figure~\ref{fig:implement}). Depending on the location of the interface, we have two kinds of patches:
\begin{itemize}
\item If a patch is not cut by the interface, {we divide} it into four quadrilaterals {$T_1,\;...\;,T_{4}$}. 
  In this case we take 
  the standard space of piecewise biquadratic functions as follows: 
  \begin{align*}\label{biquad}
    \hat{Q}=\big\{\phi\in C(\hat{P}) \, ,\, {\phi |_{\hat{T_i}}} \in \text{span}\{1,x,y,x^2,xy,y^2,xy^2,x^2y,x^2y^2\},\; i=1,...,4\big\},
  \end{align*}
  where $\hat{P}$ is the reference patch on the unit square $(0,1)^2$ consisting of the four quadrilaterals {$\hat{T}_1,...,\hat{T}_{4}$}.
\item If the patch is cut by the interface, {we divide} into eight triangles $T_1,\;...\;,T_8$. Here we define the space of piecewise quadratic functions as follows: 
  \begin{equation*}\label{quadratic}
    \hat{Q}=\big\{\phi\in C(\hat{P}) \, ,\, \phi |_{\hat{T_i}} \in \text{span}\{1,x,y,x^2,xy,y^2\},\;  i=1,...,8\big\}, \;\;\; 
  \end{equation*}
  where the reference patch $\hat{P}$ consists of eight triangles $\hat{T}_1,...,\hat{T}_8$. 
\end{itemize}
In both cases, we have locally 25 basis functions in each patch (see Figure~\ref{fig:implement})
\begin{equation*}
  Q(P):=\text{span}\{\phi_i\}, \,\, \phi_i:=\hat{\phi_i}\circ {\hat{\xi}^{-1}_P},\,\,\, i=1,...,25.
\end{equation*}
{${\hat{\xi}_P}\in \hat{Q}$} is the reference patch map, which is defined in an isoparametric way by
\begin{equation}\label{dofs}
  {\hat{\xi}_P}(\hat{\vec{x}}):=\sum_{j=1}^{25}y_j \,\hat{\phi_j}
\end{equation}
{for the 25 vertices $y_i, i =1,...,25$ of $P$}.

\subsubsection{Definining the patch type and movement of mesh nodes}
We assume that the interface is given as zero level-set of an implicit level-set function $\gamma (\vec{x})$
\begin{align*}
  \gamma(\vec{x}) = 0 \qquad \Leftrightarrow\qquad \vec{x}\in \Gamma.
\end{align*}
The patch type and the edges that are cut {are determined by the sign of} $\gamma(\vec{x}_i)$ in the exterior vertices $\vec{x}_1,...,\vec{x}_4$, see Figure~\ref{fig:implement}. 
An edge $e$ is cut, if $\gamma(\vec{x}_1)\cdot \gamma(\vec{x}_2)<0$ for its two end points $\vec{x}_1, \vec{x}_2$. The intersection of the interface with the edge can the be found
by applying Newton's method locally to find the zero $r$ of
\begin{align}\label{gammar}
  \gamma \big(\vec{x}_1+r(\vec{x}_2-\vec{x}_1)\big)=0,
\end{align}
see Figure \ref{fig:implement} (a). The edge midpoints $\vec{e}_1$ and $\vec{e}_2$ will be moved to the
respective 
position $\vec{x}_1+r(\vec{x}_2-\vec{x}_1)$. Next, we define a preliminary coordinate for the midpoint 
of the patch $\vec{x}_m$ as the midpoint of a segment $\vec{e}_1\vec{e}_2$, see Figure \ref{fig:implement} (b).
For a second-order interface approximation, it is necessary to move $\vec{x}_m$ to the interface $\Gamma$ in the configurations $A$ to $D$. We use again
Newton's method to move $\vec{x}_m$ to the interface along a normal line, see Figure \ref{fig:implement} (c). 
Second, we also move the midpoints of the segments $\vec{e}_1\vec{x}_m$ and $\vec{x}_m\vec{e}_2$ analogously, see Figure \ref{fig:implement} (d). Finally, we need to specify 
a criteria to ensure that the resulting sub-triangles with curved boundaries fulfill a maximum angle condition. Details are given in appendix C.

\begin{remark} 
  A disadvantage of the modified second-order finite element method described above is that the stiffness matrix can be ill-conditioned for certain anisotropies. In particular, the condition number depends not only on the mesh size, but also on how the interface intersects the triangulation (e.g., $s,r\rightarrow 0$). In section~\ref{sec.num} we consider two examples, where the condition number 
  of the stiffness matrix is not bounded. 
  For this reason a hierarchical finite element 
  basis was introduced in \cite{Frei1} for linear finite elements and it was shown that {the condition number of} the stiffness matrix satisfies the usual bound ${\cal O}(h_P^{-2})$ with a constant 
  that does not depend on the position of the interface. We extend this approach to the second-order finite element method below. 
  We will see that
  the condition number for a scaled hierarchical basis is reduced significantly, although we can not guarantee the optimal bound for the method presented here.    
\end{remark}

{\begin{remark}[Comparison with unfitted finite element methods]
In contrast to unfitted finite element methods (for example Hansbo\& Hansbo~\cite{HansboHansbo2002}), 
continuity can be imposed strongly within the finite element spaces in the locally finite element method, while in unfitted methods a weak imposition based on Nitsche's method is typically used. Thus, an advantage of the locally modified finite element 
method is that it is parameter-free.
 The most tidious task in the implementation of unfitted finite element methods is the construction of suitable quadrature formulas. Usually, the cut cells are sub-divided into sub-cells~\cite{CutFEM}, similarly to the subdivision used within the locally modified finite element method. For the purpose of quadrature no maximum angle condition is needed, which is required for the fitted method. This might be considered as an advantage of the unfitted approach, in particular concerning three dimensional problems. As a remeady in the fitted method, one could allow to move exterior patch vertices in certain "pathological" situations, as discussed in Remark~\ref{rem.3d}.
\end{remark}}

\section {Numerical examples}
\label{sec.num}
  
  The higher order parametric finite element method is based on the
  finite element framework \emph{Gascoigne 3d}\cite{gascoigne}. The
  source code is freely available at \url{https://www.gascoigne.de}
  {and published as Zenodo repository~\cite{ZenodoGascoigne2021}.} 
  For reproducibility of the numerical results, the following two
  configurations are implemented and described in a separate {Zenodo
  repository~\cite{Zenodo2021}.}

\subsection {Example 1}
We consider a square domain $\Omega = (-2,2)^2$. The domain is split into two domains $\Omega_1$ and $\Omega_2$ by the 
{interface $\Gamma {=} \{(x,y)\in \Omega \, | \, l(x,y)=0\}$ with level-set function} $l(x,y)= y-2(x+\delta h)^2+0.5$, where $\delta \in [0,1]$ and $h$ is the mesh size. We take $\nu_1=4$ and $\nu_2=1$ and choose the exact solution as
\[
u(x,y)=
\begin{cases}
  \frac{1}{\nu_1}\sin(l), \;\; \text{in} \;\,\Omega_1, \\
  \frac{1}{\nu_2}\sin(l), \;\; \text{in} \;\, \Omega_2,
\end{cases}
\]
{by setting} the right-hand side $f_i=-\nu_i \Delta u$ and Dirichlet boundary data {accordingly}. We vary $\delta\in [0,1]$, such that
this example includes different configurations with arbitrary anisotropies. The subdomains and the exact solution for this example are shown in Figure \ref{fig:problem_1}. 
\begin{figure}[h]
  \centering
  \begin{minipage}{4cm}
    \hspace{-3cm}
    \includegraphics[trim=0mm 90mm 30mm 40mm,clip,width=1.2\textwidth]{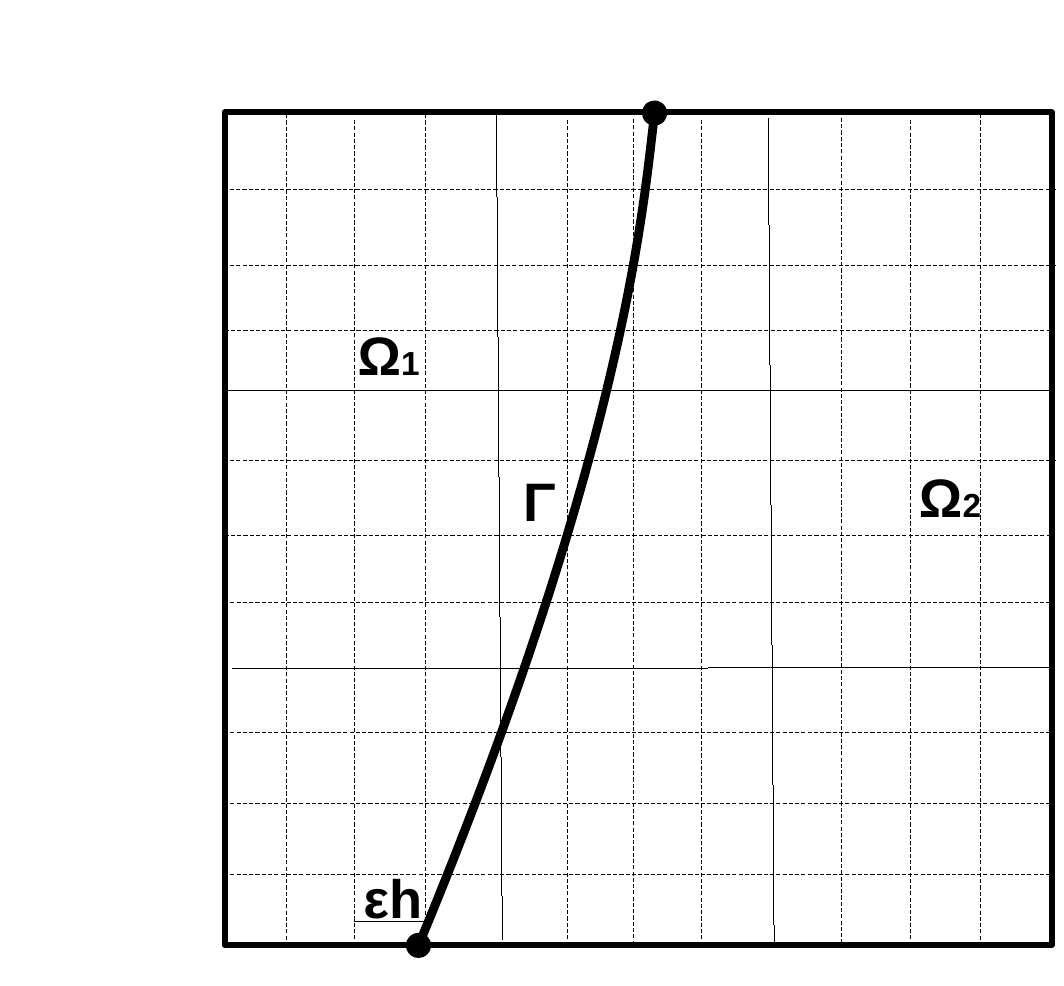}
  \end{minipage}
  \hspace{1cm}
  \begin{minipage}{4cm}
    \includegraphics[trim=40mm 10mm 20mm 10mm,clip,width=2.3\textwidth]{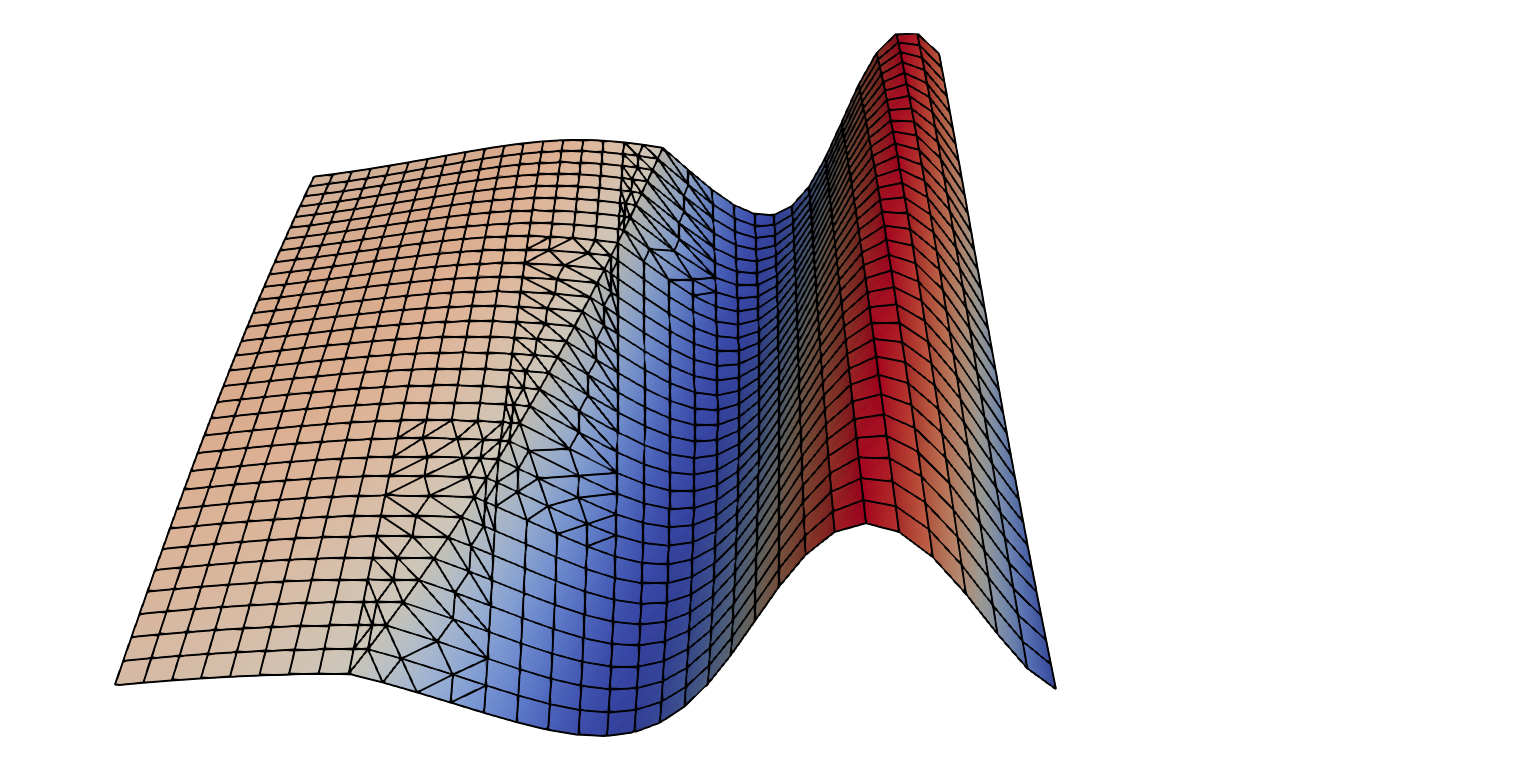}
  \end{minipage}
  \caption{Example 1. \textit{Left}: Configuration of the test problem. \textit{Right}: Sketch of the exact solution}
  \label{fig:problem_1}
\end{figure}
In this example the interface could be resolved with second order on all {refinement levels} ($n_l=0$). 
Table~\ref{tab.1} shows the {discrete energy norm error $\vertiii{u-u_h}$ and the $L^2$-norm error} as well as estimated convergence orders on several levels of global mesh refinement for the fixed parameter $\delta=0$.
According to the a priori error 
estimate in Theorem~\ref{t10}, we observe fully quadratic convergence in the {discrete} energy norm and fully cubic convergence in the $L^2$ - norm.  
\begin{table}[h]
  \centering 
  \begin{tabular}{ c | c  c | c  c }
    \toprule
    $h$		& $L^2$ - error           & $EOC$         & energy error     	  & $EOC$        \\  \midrule
    \rule{0pt}{2.8ex}
    $1/32$     	& $1.74\cdot 10^{-4}$    & -      	  & $2.08 \cdot 10^{-2}$   & -            \\   
    \rule{0pt}{3ex}
    $1/64$    	& $2.13\cdot 10^{-5}$   & $3.023$        & $5.22\cdot 10^{-3}$     & $1.998$        \\   
    \rule{0pt}{3ex}
    $1/128$    	& $2.65\cdot 10^{-6}$   & $3.006$       & $1.31\cdot 10^{-3}$     & $1.999$        \\   
    \rule{0pt}{3ex}
    $1/256$    	& $3.31\cdot 10^{-7}$   & $3.004$       & $3.26\cdot 10^{-4}$     & $2.000$       \\
    \bottomrule
  \end{tabular} 
  
  \caption{Example 1. Errors in the $L^2$ - norm and the {discrete} energy norm, including an estimated order of convergence which is computed from two consecutive values in each row for Example 1 and $\delta=0$.\label{tab.1}}
\end{table}
{In Figure \ref{fig:error}}, we plot the discrete energy norm error and the $L^2$-norm error for $\delta \in [0,1]$ on several levels of global mesh refinement and observe that the error is bounded independently of $\delta$.

In Figure \ref{fig:condition-numberParabel}, we show how the condition number depends on the parameter $\delta \in [0,1]$ by moving the interface. We get the largest condition numbers at $\delta = 0.84$. Furthermore, we show a zoom-in of the numbers for $\delta \in {[0.83,0.85]}$ in Figure \ref{fig:condition-numberParabel}, right. We see that the condition number {is reduced by a factor of $100$ using a scaled hierarchical basis}, but that is not necessarily bounded for arbitrary anisotropies.    
\begin{figure}[h]
  \hspace{1cm}
  \begin{minipage}{5cm}
    \includegraphics[trim=0mm 0mm 0mm 0mm,clip,width=1.4\textwidth]{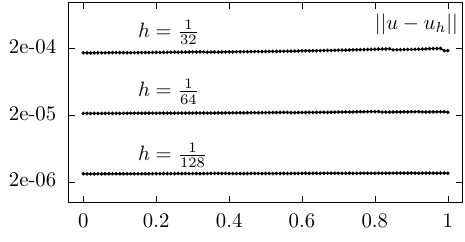}
  \end{minipage}
  \hspace{2.2cm}
  \begin{minipage}{5cm}
    \includegraphics[trim=0mm 0mm 0mm 0mm,clip,width=1.4\textwidth]{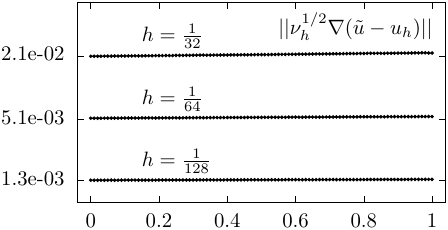}
  \end{minipage}
  \caption{Example 1. $L^2$ - norm and {discrete} energy norm
    errors for Example 1 with $x=1.0+\delta h$ and $\delta \in
    [0,1].$} 
  \label{fig:error}
\end{figure}
%
\begin{figure}[!h]
  \hspace{0.8cm}
  \begin{minipage}{5cm}
    \includegraphics[trim=0mm 0mm 0mm 0mm,clip,width=1.4\textwidth]{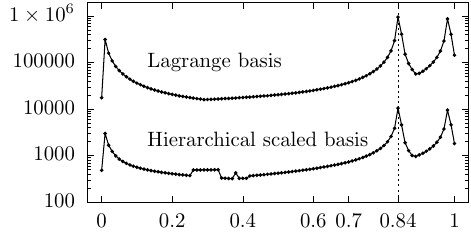}
  \end{minipage}
  \hspace{2.2cm}
  \begin{minipage}{5cm}
    \includegraphics[trim=0mm 0mm 0mm 0mm,clip,width=1.4\textwidth]{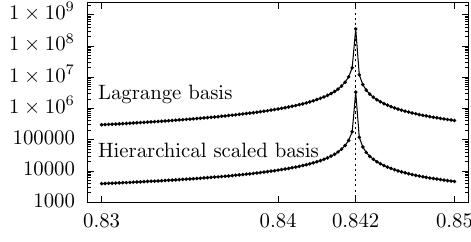}
  \end{minipage}
  \caption{Example 1. Condition number of the stiffness matrix
    depending on the position of the interface $\delta$. Comparison of
    the standard Lagrangian basis and a scaled hierarchical basis for
    $h=1/32$. \textit{Left}: $\delta \in [0,1]$. \textit{Right}:
    Zoom-in for $\delta \in [0.83,0.85]$.}
  \label{fig:condition-numberParabel}
\end{figure}
\subsection{Example 2}

We consider a square domain $\Omega = (-2,2)^2$ that is split into a ball $\Omega_1=B_r(x_0,y_0) $ with $r=0.3$ and $(x_0,y_0)=(1+\delta h,1.2)$, where $\delta \in [0,1]$, and $\Omega_2=\Omega \setminus \bar{\Omega}_1$. We take the exact solution as in example 1, with the level set function replaced by $l(x,y)=(x-x_0)^2+(y-y_0)^2-r^2$. In Figure \ref{fig:exact} we show the configuration and the exact solution of this example. 
For different $\delta\in[0,1]$, this example includes {all configurations A-E introduced above} with different anisotropies.
\begin{figure}[h]
  \hspace{1.2cm}
  \begin{minipage}{4cm}
    \includegraphics[trim=0mm 95mm 0mm 50mm,clip,width=1.4\textwidth]{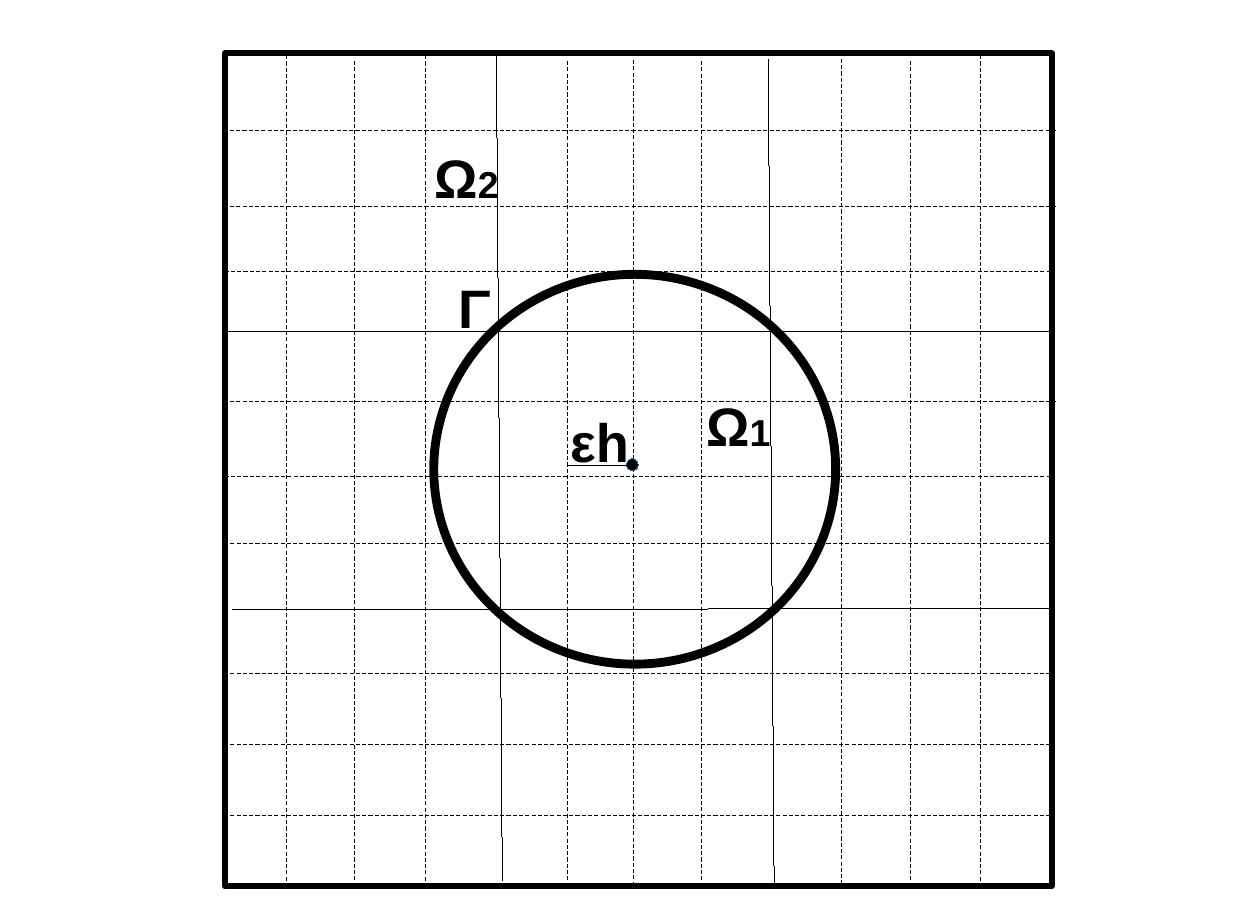}
  \end{minipage}
  \hspace{2.5cm}
  \begin{minipage}{4cm}
    \includegraphics[trim=15mm 5mm 25mm 30mm,clip,width=2.1\textwidth]{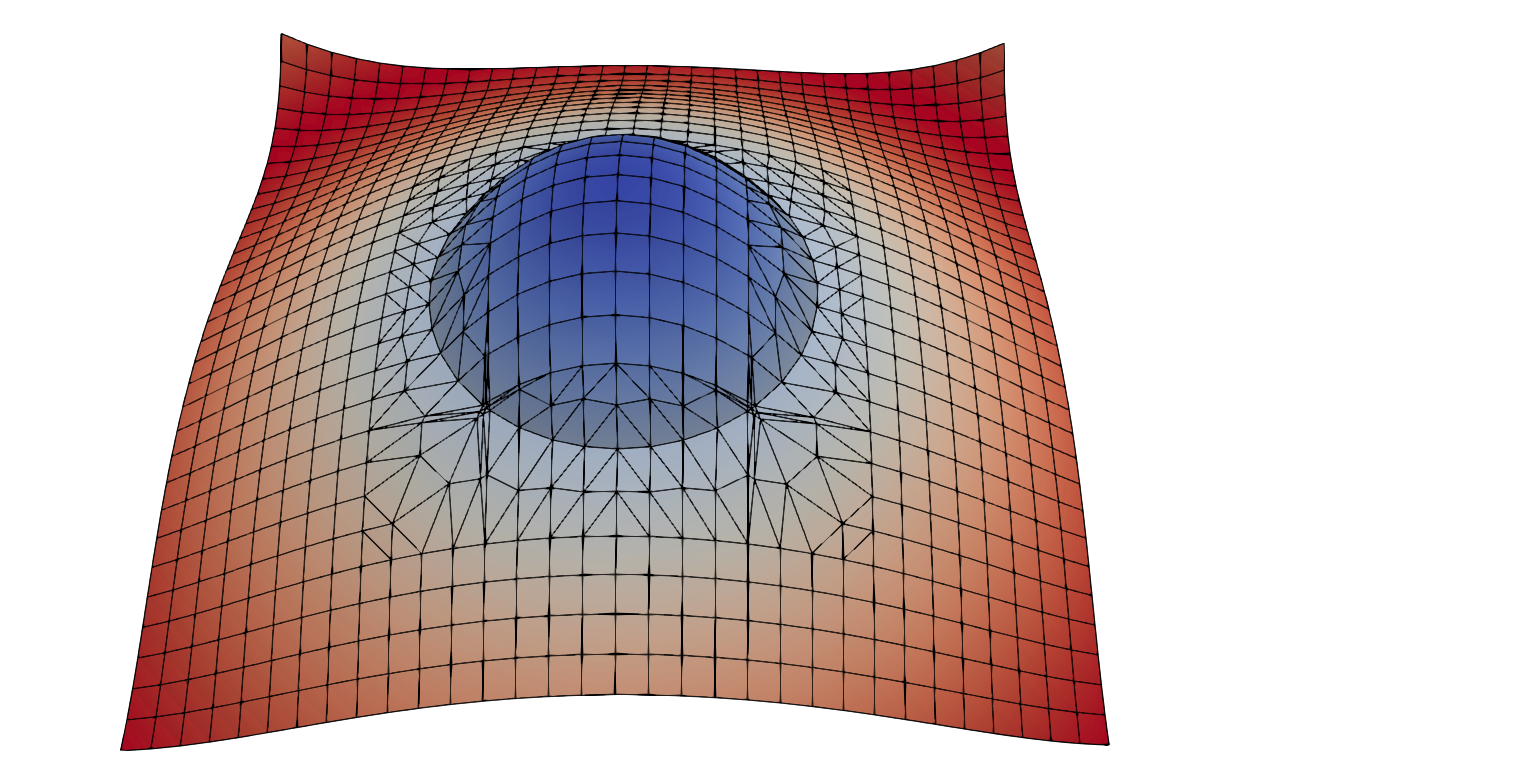}
  \end{minipage}
  \caption{Example 2. \textit{Left}: Configuration of the test problem. \textit{Right}: Sketch of the exact solution}
  \label{fig:exact}
\end{figure}

The $L^2-$norm and the {discrete} energy norm errors are shown in Figure \ref{fig:error-kreis} for $\delta \in [0,1]$ on several levels of global mesh refinement. We observe convergence in both norms for $\delta>0$. The errors vary slightly depending on $\delta$. Its magnitude depends mainly on the number of linearly approximated elements ($n_l$):
We have $n_l= 0$ for $\delta=0$ on all mesh levels, while $n_l>0$ for all other values of $\delta$. We observe that the errors increase from $\delta=0$ to $\delta=0.01$, as $n_l$ increases from 0 to 8. Moreover, the number of linearly approximated elements increases for $h=1/64$ once more, from $n_l=8$ to $n_l=16$ in the range $\delta \in [0.74,0.81]$. Again, we observe a slight increase in the magnitude of the error within this range. This indicates that the constant $c_l n_l^{1/2}$ corresponding to the linearly approximated part in \eqref{energyW2infty} is larger than the constant $c_q$ arising from the quadratically approximated elements.


Table~\ref{tab.2}, Table~\ref{tab.3} and Table~\ref{tab.4} show the $L^2$-norm and the {discrete} energy norm errors obtained on several levels of global mesh refinement for the fixed positions $x_0=1.0+\frac{\delta_0}{64}$ of the midpoint, with $\delta_0\in\{0,\,0.01,\, 0.8\}$, which results in three different cases ($n_l=0, n_l=8$ and $n_l=16$) for $h=1/64$.

In Table~\ref{tab.2} ($\delta_0=0$) we observe fully quadratic (resp. cubic convergence) in the {discrete} energy norm (resp. the $L^2$-norm) 
as shown in Theorem~\ref{t10}, as no linearly approximated elements are present. 
 This changes slightly for the other values of $\delta_0$, see Table~\ref{tab.3} and Table~\ref{tab.4}. 
 
 In Table~\ref{tab.3} ($\delta_0=0.01$), we see that 8 linearly approximated elements were required on all mesh levels. The convergence order in the {discrete} energy norm seems to be fully quadratic {(according to \eqref{energyW2infty}), while in the $L^2$-norm error the
logarithmic factor $|\ln(h)|^{1/2}$ leads to a slightly reduced convergence, as predicted in  Theorem~\ref{t10}.}
 
 For $\delta_0=0.8$, the number $n_l$ increases from 8 to 16 between the coarsest and the second-coarsest refinement level and stays constant from then, see Table~\ref{tab.4}. This is again reflected in the magnitude of the error: The reduction factor between the coarsest mesh levels lies below 4 in the energy norm, and below 8 in the $L^2$-norm error, which shows again that {the term $c_l n_l^{1/2} |\ln(h)|^{1/2}$ in front of the linearly approximated part is larger than the constant $c_q$ in front of the quadratic counterpart. On the remaining mesh levels, the estimated convergence is again fully quadratic in the {discrete} energy norm and due to the logarithmic term slightly below 3 in the $L^2$-norm, in agreement with Theorem~\ref{t10}.}
  
 For $\delta_0 = 0.8$ and $h_P=\frac{1}{32}$, we show the resulting finite element mesh in Figure \ref{fig:eps_08-3}, where in 8 of the 18 patches, which are cut by the interface, a linear approximation was required, including a zoom around one linearly approximated patch on the right.
\begin{figure}[h]
  \hspace{1.2cm}
  \begin{minipage}{5cm}
    \includegraphics[trim=0mm 0mm 0mm 0mm,clip,width=1.4\textwidth]{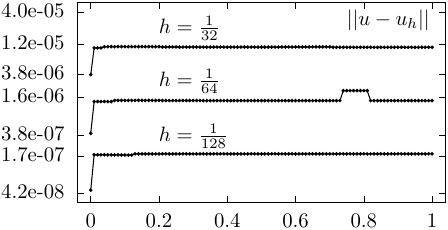}
  \end{minipage}
  \hspace{2.2cm}
  \begin{minipage}{5cm}
    \includegraphics[trim=0mm 0mm 0mm 0mm,clip,width=1.4\textwidth]{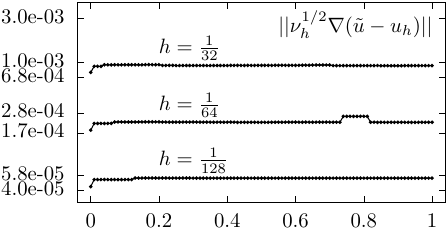}
  \end{minipage}
  \caption{Example 2. $L^2$ - norm and {discrete} energy-norm errors depending on $x=1.0+\delta h$ with $\delta \in [0,1].$}
  \label{fig:error-kreis}
\end{figure}
\begin{table}[h]
  \centering 
  \begin{tabular}{ c | c  c | c  c }
    \toprule
    $h$	& $L^2$ - error            & $EOC$        & energy error          & $EOC$       	\\  \midrule
    \rule{0pt}{2.8ex}
    $1/32$ &  $3.44\cdot 10^{-6}$     &  -           &  $4.36\cdot 10^{-4}$   &  -              \\   
    \rule{0pt}{3ex}
    $1/64$ &  $3.85\cdot 10^{-7}$     &  $3.159$     &  $9.69\cdot 10^{-5}$   &  $2.170$      \\   
    \rule{0pt}{3ex}
    $1/128$ &  $4.56\cdot 10^{-8}$    &  $3.078$     &  $2.28\cdot 10^{-5}$   &  $2.085$        	\\   
    \rule{0pt}{3ex}
    $1/256$ &  $5.54\cdot 10^{-9}$    &  $3.040$     &  $5.52\cdot 10^{-6}$   &  $2.048$        	\\   \bottomrule
  \end{tabular} 
  \caption{Example 2. $L^2$ -norm and {modified} energy norm errors, and convergence order for $\delta_0 = 0$ ($n_l=0$).\label{tab.2}}
\end{table}
\begin{table}[h]
  \centering 
  \begin{tabular}{c | c  c | c  c | c | c }
    \toprule
    $h$	& $L^2$ - error          & $EOC$        & energy error          & $EOC$      & $PN$	 &$n_l$ \\  \midrule
    \rule{0pt}{2.8ex}
    $1/32$  &  $1.05\cdot 10^{-5}$   &  -           &  $9.00\cdot 10^{-4}$    &  -         &  $18$	 &  $8$    \\   
    \rule{0pt}{3ex}
    $1/64$  &  $1.37\cdot 10^{-6}$   &  $2.943$     &  $2.15\cdot 10^{-4}$    &  $2.066$    &  $36$  &  $8$   \\   
    \rule{0pt}{3ex}
    $1/128$  &  $1.81\cdot 10^{-7}$  &  $2.920$     &  $5.27\cdot 10^{-5}$     &  $2.028$    &  $76$  &  $8$    \\   
    \rule{0pt}{3ex}
    $1/256$  &  $2.39\cdot 10^{-8}$  &  $2.921$     &  $1.30\cdot 10^{-5}$     &  $2.015$    &  $154$ &  $8$    \\   \bottomrule
  \end{tabular} 
  \caption{Example 2. $L^2$- and {discrete} energy - norm errors for $\delta_0 = 0.01$, including estimated convergence orders obtained from two consecutive values. $PN$ denotes the number of patches which are cut by the interface and $n_l$ the number of the linear approximated elements.\label{tab.3}}
\end{table}
\begin{table}[!h]
  \centering 
  \begin{tabular}{ c | c c | c c | c |  c}
    \toprule
    $h$	& $L^2$ - error          & $EOC$          &energy error          & $EOC$       & $PN$	 &$n_l$	\\  \midrule
    \rule{0pt}{2.8ex}
    $1/32$  &  $1.09\cdot 10^{-5}$   &  -            &  $9.19\cdot 10^{-4}$   &  -         &  $18$	 &  $8$     \\   
    \rule{0pt}{3ex}
    $1/64$  &  $2.07\cdot 10^{-6}$   &  $2.390$      &  $2.57\cdot 10^{-4}$   &  $1.839$    &  $36$  &  $16$  \\   
    \rule{0pt}{3ex}
    $1/128$ &  $2.73\cdot 10^{-7}$   &  $2.924$      &  $6.35\cdot 10^{-5}$   &  $2.016$    &  $76$  &  $16$    \\   
    \rule{0pt}{3ex}
    $1/256$ &  $3.59\cdot 10^{-8}$   &  $2.927$      & $ 1.58\cdot 10^{-5}$   &  $2.007$    &  $152$ &  $16$   \\   \bottomrule
  \end{tabular} 
  \caption{Example 2. $L^2$ and {discrete} energy - norm errors, including an estimated convergence order for $\delta = 0.8$.\label{tab.4}}
\end{table}
\begin{figure}[h]
  \hspace{1.2cm}
  \begin{minipage}{5cm}
    \includegraphics[trim=110mm 0mm 115mm 0mm,clip,width=1.1\textwidth]{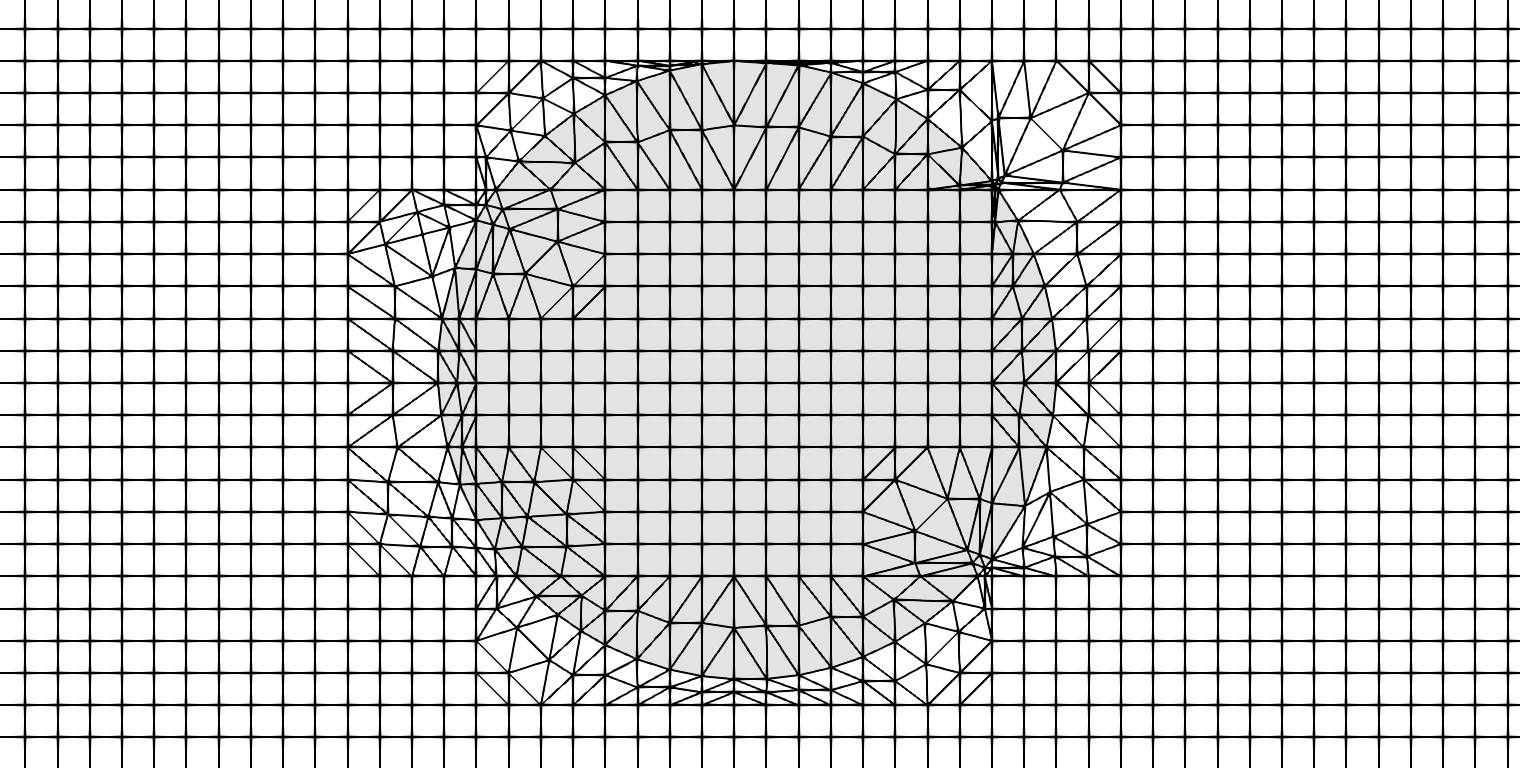}
  \end{minipage}
  \hspace{2.1cm}
  \begin{minipage}{5cm}
    \centering
    \includegraphics[trim=0mm 0mm 90mm 30mm,clip,width=1.4\textwidth]{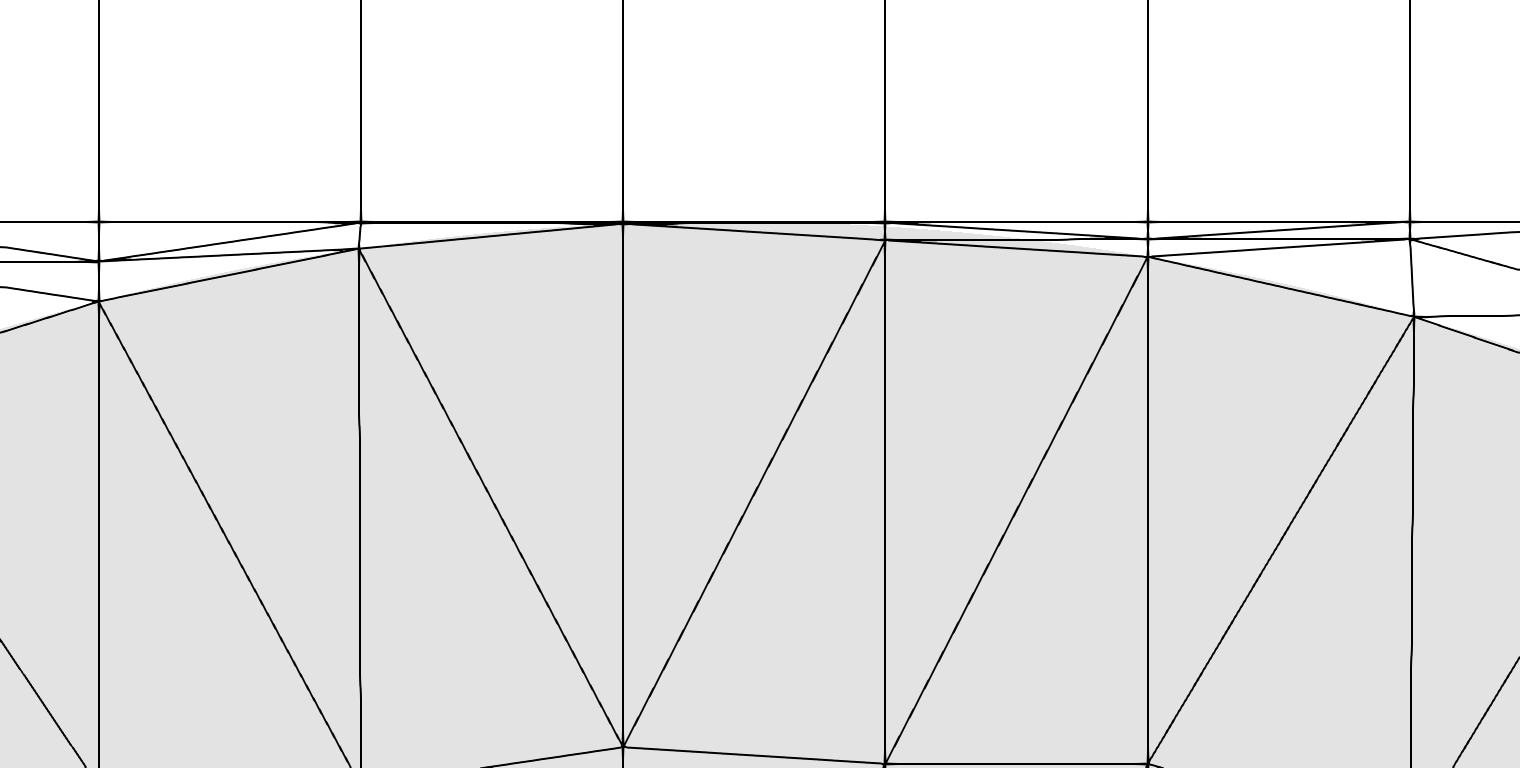}
  \end{minipage}
  \caption{Example 2. \textit{Left}: Illustration of the sub-elements for $h=1/32$ and $\delta = 0.8$. \textit{Right}: Zoom of the upper part with  linearly approximated elements (top right).}
  \label{fig:eps_08-3}
\end{figure}
%
%

In Figure \ref{fig:condition-numberKreis1}, we show how the condition numbers depend on the parameter $\delta \in [0,1]$ when moving the interface. We get the largest condition numbers at $\delta \approx 0.04$ for $h=1/32$ and at $\delta \approx 0.07$ for $h=1/64$, respectively. The condition numbers are again reduced by a factor of approx.\,100 for the scaled hierarchical basis compared to the standard Lagrangian basis. 
 
\begin{figure}[h]
  \hspace{1.1cm}
  \begin{minipage}{5cm}
    \includegraphics[trim=0mm 0mm 0mm 0mm,clip,width=1.4\textwidth]{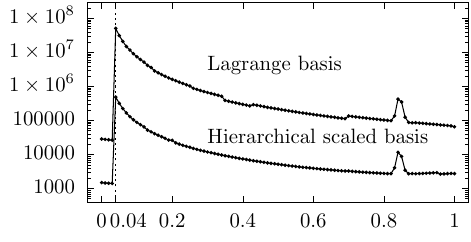}
  \end{minipage}
  \hspace{2.2cm}
    \begin{minipage}{5cm}
    \includegraphics[trim=0mm 0mm 0mm 0mm,clip,width=1.4\textwidth]{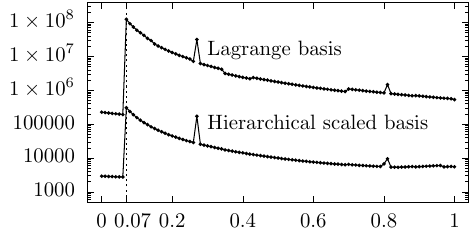}
  \end{minipage}
  \caption{Example 2. Condition number of the stiffness matrix depending on the displacement of the circle. Comparison of the Lagrange and hierarchical scaled basis for $h=1/32$ (left) and $h=1/64$ (right).}
  \label{fig:condition-numberKreis1}
\end{figure}
%
\section{Conclusion}
We have presented an extension of the locally modified finite element method for interface problems introduced in \cite{Frei1}, to second order. {We were able to show optimal-order error estimates of order two in a discrete energy norm and almost of order three (up to a logarithmic term) in the $L^2$-norm.} Finally, we have presented different numerical examples that illustrate the convergence behaviour and the performance of the method. In future, we plan to {extend the method} to \emph{inf-sup} stable finite elements for the discretization of interface problems including the Stokes- and Navier-Stokes equations.


\section*{Appendix A: Linear interface approximation}

We distinguish between the following five cases, see Figure~\ref{fig:config}
\begin{itemize}
\item Configuration A: The patch is cut in two opposite nodes.
\item Configuration B: The patch is cut at the interior of one edge and in one node.
\item Configuration C: The patch is cut at the interior of two opposite edges.
\item Configuration D: The patch is cut at the interior of two adjacent edges with \\
  $r\in (0,\frac{1}{2}),\; s\in (\frac{1}{2},1)$.
\item Configuration E: The patch is cut at the interior of two adjacent edges with 
  \begin{itemize}
  \item $r\in (0,1)$ and $s\in (0,\frac{1}{2})$
  \item $r\in (\frac{1}{2},1)$ and $s\in (0,1)$
  \end{itemize}
\end{itemize}

The subdivisions can be anisotropic with the parameters $r,s \in(0,1)$ in the configurations $B$, $C$, $D$ and $E$. These parameters describe the relative position of the intersection points with the interface on the edges. We denote by $\vec{e}_i\in \mathbb{R}$, $i=1,2,3,4,$ the vertices on the edge. When the interface intersects an edge, we move the corresponding point $\vec{e}_i$, $i=1,...,4$ on the intersected edge to the point of the intersection (see Figure \ref{fig:config}). If an edge is not intersected by the interface, we take $\vec{e}_i$ as midpoint of this edge. 
By $\vec{x}_m\in \mathbb{R}^2$ we denote the midpoint of the patch, which has different positions depending on the configurations. {Precisely, it is chosen} as intersection of the line connecting $\vec{e}_1$ and $\vec{e}_3$ with the line connecting $\vec{e}_2$ and $\vec{e}_4$ for configurations $A$, $C$ and $E$. For configuration $B$ we choose the midpoint as intersection of the line connecting $\vec{e}_1$ and $\vec{e}_3$ with the line connecting $\vec{x}_1$ and $\vec{e}_2$. The midpoint 
for the configuration $D$ can be chosen as midpoint of the line segment $\vec{e}_1\vec{e}_2$. 

{In all configurations the patch is first} divided into four quadrilaterals. We note that each of these has at least one right angle. {The sub-quadrilaterals are then further divided into two triangles by either resolving the interface with an interior mesh line or -if this is not necessary- by splitting the largest interior angle of the quadrilateral.}


\section*{Appendix B: Proof of Lemma \ref{linear}}
\begin{proof}
  {
    First, the patch is split into four sub-quadrilaterals $K_1,...,K_4$, each of which is then split into two triangles. If we can show that all angles of the quadrilaterals 
    are bounded by $135^\circ$, this applies for the sub-triangles as well.
    Moreover, if the splitting into triangles in $K_i$ for some $i=1,...,4$ is not 
    determined by the interface position, we split in such a way that the largest angle of $K_i$ is divided. In this case 
    the bound for the maximum angles of the sub-triangles can be further improved.

    We consider the {configurations A-E} shown in Figures~\ref{fig:config} separately. In all cases the angles at the vertices $\vec{x}_i, i=1,...,4$ are exactly $90^{\circ}$
    and the angles at the edge midpoints $\vec{e}_i, i=1,...,4$ lie between $45^\circ$ and $135^{\circ}$ 
    (Note that this bound is not sharp, if we divide in an optimal way into sub-triangles).

    In configuration $A$ we have two squares and four right-angled triangles. This case is obvious and the maximum angle 
    of the sub-triangles is $90^\circ$.

    In configuration B and C each quadrilateral $K_i (i=1,...,4)$ has two right angles, as the positions of $\vec{e}_1$ and $\vec{e}_3$, or $\vec{e}_2$ and $\vec{e}_4$, respectively, are fixed. Let us 
    consider examplarily the configuration shown in 
    Figures~\ref{fig:config}(b). As mentioned above, the angles in $\vec{e}_2$ and $\vec{e}_4$ lie between $45^\circ$ and $135^\circ$. By symmetry the angles around $\vec{x}_m$ are exactly the same. 
    Therefore, in 
    both configurations all angles of the sub-triangles lie below $135^\circ$.

    In configuration $D$, we get one degenerate quadrilateral $K_2$ with a maximum angle of $180^\circ$. As this angle is divided by connecting $\vec{x}_m$ and $\vec{x}_2$
    we have the following bounds for the triangles of $K_2$
    \[
    \begin{aligned}
      \cos(\angle \vec{e}_1\vec{x}_m\vec{x}_2)&= \frac{(\vec{e}_1-\vec{x}_m)\cdot (\vec{x}_2-\vec{x}_m)}{|\vec{e}_1-\vec{x}_m|\cdot|\vec{x}_2-\vec{x}_m|}
      =\frac{((r-1),-s)\cdot ((1-r),-s)}{(1-r)^2+s^2 }\\
      &=\frac{-(1-r)^2+s^2}{(1-r)^2+s^2}\in \Big(-\frac{3}{5}, \frac{3}{5}\Big).
    \end{aligned}
    \]
    such that $\angle \vec{e}_1\vec{x}_m\vec{x}_2 \in (53^{\circ},127^{\circ})$. The other angles in $\vec{x}_m$ are bounded above by $90^\circ$.
    The angles in $\vec{e}_1,...,\vec{e}_4$ are again bounded by $135^\circ$.

    In configuration E, the angles in $\vec{e}_1,...,\vec{e}_4$ are all between $63^\circ$ and $117^\circ$. A bound on the angles of the quadrilaterals at $\vec{x}_m$ is therefore 
    given by $360^\circ - 2\cdot 63^\circ - 90^\circ = 144^\circ$. This maximum is attained for $r\to 1, s\to 0$ (cf. Figure~\ref{fig:config} (f)). 
    The bound is further improved, as in $K_1, K_2$ and $K_3$ the largest angles are divided when splitting into sub-triangles, resulting in angles below $90^\circ$.
    For the angle {of the subtriangle} $T_1$ at $\vec{x}_m$ we have 
    \[
    \begin{aligned}
      \cos(\angle \vec{e}_1\vec{x}_m\vec{e}_2)&= \frac{(\vec{e}_1-\vec{e}_3)\cdot (\vec{e}_2-\vec{e}_4)}{|\vec{e}_1-\vec{e}_3|\cdot|\vec{e}_2-\vec{e}_4|}
      =\frac{(r-1/2,-1)\cdot (1,s-1/2)}{\sqrt{1+(r-1/2)^2}\cdot \sqrt{1+(s-1/2)^2} }\\
      &=\frac{r-s}{\sqrt{1+(r-1/2)^2}\cdot \sqrt{1+(s-1/2)^2}}\in \Big(-\frac{1}{\sqrt{5}}, \frac{4}{5}\Big)
    \end{aligned}
    \]
    such that $\angle \vec{e}_1\vec{x}_m\vec{e}_2 \in (36^{\circ},117^{\circ})$.
  }
\end{proof}
\section*{Appendix C: Quadratic interface approximation}
For the elements with curved boundaries 
we need to ensure that all elements are allowed in the sense of Assumption~\ref{ass.allowed}
(see also Figure \ref{fig:patch}) and that the maximum angle condition shown above remains valid. 
As described in Section~\ref{sec.impl}, we move certain points to the interface in order to obtain a 
second-order interface approximation. This is possible
if the following criteria are satisified. Otherwise, we leave them 
in their original positions and obtain a first-order interface approximation in the respective element. By $\alpha_{\triangle}$
  we denote the largest angle in a triangle. 

In the first step, we move the midpoint of the patch. 
If this is possible, we shift the other corresponding points in a second step (if possible). 
We use the following criteria for each configuration.\\
\newline
\textbf{First step: Move the midpoints}
\begin{itemize}
\item \textbf{Configuration $A$}: the midpoint of the patch can be moved along the normal line $\vec{n}$ 
  (see Figure \ref{fig:config}a) if $\alpha_{\triangle} \;\leq \; \alpha_{max} \;<\; 180^{\circ}$.
\item \textbf{Configuration $B$}: the midpoint of the patch can be moved along the line segment $\vec{e}_1\vec{e}_3$, if the relative length $d=\frac{|\vec{e}_1-\vec{x}_m|}{|\vec{e}_2-\vec{e}_1|}$ 
  of the line $\vec{e}_1\vec{x}_m$ (see Figure \ref{fig:config}b) 
  satisfies $\epsilon<d<1-\epsilon$ and $\alpha_{\triangle} \;\leq \; \alpha_{max} \;<\; 180^{\circ}$.
\item \textbf{Configuration $C$}: the midpoint of the patch can be moved along the line segment $\vec{e}_2\vec{e}_4$, if the parameter $d$ (see Figure \ref{fig:config}c) 
  satisfies $\epsilon<d<1-\epsilon$. 
\item \textbf{Configuration $D$}: the midpoint of the patch can be moved along the normal line $\vec{n}$ 
  (see Figure \ref{fig:config}d) if $\alpha_{\triangle} \;\leq \; \alpha_{max} \;<\; 180^{\circ}$.
\item \textbf{Configuration $E$}: in this configuration we do not need to move the midpoint of the patch (see Figures \ref{fig:config}e), and \ref{fig:config}f).
\end{itemize}
\textbf{Second step: Move other points}
\begin{itemize}
\item In the second step, we investigate the other two points that need to be moved in order to obtain a second-order interface approximation. 
  These are the points between the midpoint of the patch and the  points where exterior edges are intersected. 
  In all configurations, we obtain triangles with one curved edge (see Figure \ref{fig:config-step2}). It can happen that this curved edge intersects other edges of the element $T$.
  Thus, we shift the corresponding points along the normal line to the interface, if and only if the curved edge of the triangle does not 
  cut any other edges and $\alpha_{\triangle} \;\leq\;\alpha_{max} \;<\; 180^{\circ}$. 
\end{itemize}
%

%

\bibliographystyle{ieeetr}

\end{document}